\documentclass[11pt,draft,reqno]{amsart}

\usepackage{amssymb} 
\usepackage{dsfont} 
\usepackage{mathrsfs} 
\usepackage{stmaryrd} 
\usepackage{color}

\DeclareMathAlphabet{\mathpzc}{OT1}{pzc}{m}{it} 

\newtheorem{Thm}{Theorem}[section]
\newtheorem{Cor}{Corollary}[section]
\newtheorem{Lem}{Lemma}[section]
\newtheorem{Prop}{Proposition}[section]

\newtheorem{Def}{Definition}[section]

\theoremstyle{definition}
\newtheorem{Rem}{Remark}[section]

\theoremstyle{definition}

\makeatletter
\@addtoreset{equation}{section}
\makeatother

\renewcommand\sp{\hspace{2.8ex}} 

\newcommand\void{\varnothing} 
\newcommand\setmeno{\!\smallsetminus\!} 
\newcommand\function{\longrightarrow} 

\newcommand\en{\mathbb{N}} 
\newcommand\ar{\mathbb{R}} 
\newcommand{\eps}{\varepsilon} 

\providecommand{\opint}[1]{\hspace{0.15ex}\left]#1\right[\hspace{0.15ex}} 
\providecommand{\clint}[1]{\hspace{0.045ex}\left[#1\right]} 
\providecommand{\clsxint}[1]{\hspace{0.1ex}\left[#1\right[\hspace{0.15ex}} 
\providecommand{\cldxint}[1]{\hspace{0.15ex}\left]#1\right]} 

\newcommand\X{\textsl{X}\hspace{0.21ex}} 
\newcommand\Y{\textsl{Y}\hspace{0.21ex}} 
\renewcommand\d{\textsl{d}} 
\DeclareMathOperator{\Int}{int} 
\DeclareMathOperator{\Cl}{cl} 

\DeclareMathOperator{\cont}{Cont}  
\DeclareMathOperator{\discont}{Discont}  
\newcommand{\Czero}{{\textsl{C}\hspace{0.18ex}}} 
\DeclareMathOperator{\Lipcost}{Lip} 
\newcommand{\Lip}{{\textsl{Lip}\hspace{0.15ex}}} 
\newcommand{\AC}{{\textsl{AC}\hspace{0.17ex}}} 

\newcommand\E{\textsl{E}\hspace{0.21ex}} 
\newcommand{\convergedeb}{\rightharpoonup} 

\newcommand{\borel}{\mathscr{B}} 
\newcommand\leb{\mathpzc{L}} 
\renewcommand{\L}{{\textsl{L}\hspace{0.17ex}}} 

\DeclareMathOperator{\pV}{V} 
\DeclareMathOperator{\V}{V} 
\newcommand{\BV}{{\textsl{BV}\hspace{0.17ex}}} 
\renewcommand\l{\textsl{l}} 
\renewcommand\r{\textsl{r}} 

\renewcommand\H{\mathcal{H}} 
\newcommand\duality[2]{\langle #1,#2 \rangle} 
\newcommand\lduality[2]{\left\langle #1,#2 \right\rangle} 
\newcommand\bduality[2]{\big\langle #1,#2 \big\rangle} 
\newcommand\norm[2]{\Vert #1\Vert_{#2}} 
\newcommand{\Conv}{\mathscr{C}} 
\DeclareMathOperator{\Proj}{Proj} 
\newcommand\K{\mathcal{K}} 
\newcommand\A{\mathcal{A}} 
\newcommand\B{\mathcal{B}} 
\newcommand\Z{\mathcal{Z}} 
\newcommand\C{\mathcal{C}} 
\newcommand\G{\mathcal{G}} 
\newcommand{\Ctilde}{\widetilde{\mathcal{C}}} 

\newcommand\hausd{\mathpzc{H}} 

\newcommand\vartot[1]{\!\left\bracevert\! #1 \!\right\bracevert\!} 
\newcommand\indicator{\mathds{1}}
\DeclareMathOperator{\de}{d \! \hspace{0.2ex}} 
\newcommand{\Step}{{\textsl{St}\hspace{0.17ex}}} 
\DeclareMathOperator{\D}{D\!} 

\newcommand{\W}{{\textsl{W}\hspace{0.17ex}}} 

\renewcommand{\P}{{\mathsf{P}}} 
\newcommand{\Stop}{{\mathsf{S}}} 
\renewcommand{\P}{{\mathsf{P}}} 
\newcommand{\Q}{{\mathsf{Q}}} 
\newcommand{\M}{{\mathsf{M}}} 

\newcommand{\utilde}{\widetilde{u}} 
\newcommand\yhat{\hat{y}} 
\newcommand{\ftilde}{\widetilde{f}} 
\newcommand\xhat{\hat{x}} 
\newcommand\what{\hat{w}} 

\definecolor{blu1}{rgb}{0.1,0.1,1}

\definecolor{verde}{rgb}{0.1,0.4,0.2}

\definecolor{pink}{rgb}{1.0, 0.33, 0.64}

\begin{document}


\title[Sweeping processes]{$\BV$-norm continuity of sweeping processes \\ driven by a set with constant shape}

\author{Jana Kopfov\'a and  Vincenzo Recupero}
\thanks{Supported by GA\v CR Grant  15-12227S  and by the institutional support
for the development of research organizations I\v C 47813059. 
V. Recupero is partially supported by the INdAM - GNAMPA Project 2016 
``Sweeping processes: teoria, controllo e applicazioni a problemi rate independent''.
V. Recupero thanks the kind hospitality of the Silesian University in Opava where part of this work has been written.
The authors would like to thank Pavel Krej\v c\'{i} for useful discussions about the subject.}

\address{\textbf{Jana Kopfov\'a}\\
Mathematical Institute of the Silesian University\\ 
Na Rybn\' i\v cku 1, CZ-74601 Opava\\ Czech Republic.
    \newline
        {\rm E-mail address:}
        {\tt Jana.Kopfova@math.slu.cz}}

\address{\textbf{Vincenzo Recupero}\\
        Dipartimento di Scienze Matematiche\\ 
        Politecnico di Torino\\
        Corso Duca degli Abruzzi 24\\ 
        I-10129 Torino\\ 
        Italy. \newline
        {\rm E-mail address:}
        {\tt vincenzo.recupero@polito.it}}




\begin{abstract}
We prove the $\BV$-norm well posedness of \emph{sweeping processes} driven by a moving convex set with constant shape, namely the  $\BV$-norm continuity of  the so called \emph{play operator} of elasto-plasticity.
\end{abstract}


\maketitle


\thispagestyle{empty}


\section{Introduction}

Mathematical models of material memory are often based on the following evolution variational inequality (cf. \cite{DuvLio76, NecHla81}). Let $\H$ be a real Hilbert space with inner product $\duality{\cdot}{\cdot}$ and 
$\Z \subseteq \H$ be a closed convex subset. Given $T > 0$ and $u : \clint{0,T} \function \H$, find 
$y : \clint{0,T} \function \H$ such that
\begin{alignat}{3}
  & \duality{z - u(t) + y(t)}{y'(t)} \le 0  & \qquad & \forall z \in \Z, \quad \text{for $\leb^1$-a.e. $t \in \clint{0,T}$}, 
     \label{var in-intro} \\
  & u(t) - y(t) \in \Z & \qquad & \forall t \in \clint{0,T}, \label{constraint}
\end{alignat}
with a given initial condition
\begin{equation}\label{initial cond-intro}
  u(0) - y(0) = z_{0} \in \Z,
\end{equation}
where $y'$ denotes the time derivative of $y$ and $\leb^1$ is the one dimensional Lebesgue measure. Variational 
inequalities of the form \eqref{var in-intro}--\eqref{initial cond-intro} play an important role in elasto-plasticity, 
ferromagnetism, and phase transitions. It is worth noting that in the new unknown function $x := u - y$, inequality 
\eqref{var in-intro} can be equivalently formulated as the first order differential inclusion
\begin{equation}
   x'(t) + \partial I_\Z(x(t)) \ni u'(t) \qquad \text{for $\leb^1$-a.e. $t \in \clint{0,T}$}, \label{gradient flow for x - intro} 
\end{equation}
$\partial I_\Z$ being the subdifferential of the indicator function $I_\Z$:  $I_{\Z}(x) := 0$ if $x \in \Z$, $I_{\Z}(x) := \infty$ otherwise (precise definitions and formulations will be given in Sections \ref{S:Preliminaries} and 
\ref{S:state main result}). Problem \eqref{gradient flow for x - intro}--\eqref{initial cond-intro} can be solved by using classical tools from the theory of evolution equations governed by maximal monotone operators (cf. \cite{Bre73}). In particular it turns out that for every $u \in \W^{1,1}(0,T;\H)$, the space of $\H$-valued absolutely continuous maps, there exists a unique $y \in \W^{1,1}(0,T;\H)$ satisfying \eqref{var in-intro}--\eqref{initial cond-intro} almost everywhere. The resulting solution operator $\P : \W^{1,1}(0,T;\H) \function \W^{1,1}(0,T;\H) : u \longmapsto y$ is also called \emph{(vector) play operator} following \cite[p. 6, p. 151]{KraPok89} (see also \cite[p. 294]{Mor74}). The suggestive terms \emph{input} and \emph{output} are used for $u$ and $y$ respectively. On the other hand inequality 
\eqref{var in-intro} can also be interpreted as the time dependent gradient flow
\begin{equation}
  y'(t) + \partial I_{u(t) - \Z}(y(t)) \ni 0 \qquad \text{for $\leb^1$-a.e. $t \in \clint{0,T}$}. \label{play-sweep incl intro}
\end{equation}
This is a particular case of \emph{sweeping process}, which can be described as follows. Let us denote by $\Conv_\H$ the family of nonempty convex closed subsets of $\H$ and endow it with the Hausdorff metric: given $y_0 \in \H$ and 
$\C \in \AC(\clint{0,T};\Conv_\H)$, the space of $\Conv_\H$-valued absolutely continuous maps, find a function 
$y \in \W^{1,1}(0,T;\H)$ such that
\begin{alignat}{3}
  & y(t) \in \C(t) & \qquad & \forall t \in \clint{0,T}, \label{sweep constr intro} \\
  & y'(t) + \partial I_{\C(t)}(y(t)) \ni 0 & \qquad & \text{for $\leb^1$-a.e. $t \in \clint{0,T}$}, \label{sweep incl intro} \\
  & y(0) = \Proj_{\C(0)}(y_0), \label{sweep in. cond. intro}
\end{alignat}
where $\Proj_{\K}$ denotes the projection operator on a closed convex set $\K$. This problem was studied and solved in \cite{Mor71, Mor72, Mor74}. More generally in \cite{Mor77} the important case when 
$\C \in \BV^{\r}(\clint{0,T};\Conv_\H)$, the space of right continuous maps with bounded variation, is considered. In this case the formulation has to be generalized and one has to find $y \in \BV^{\r}(\clint{0,T};\H)$, a right continuous
$\H$-valued function  of bounded variation, such that there exists a positive measure $\mu$ and a $\mu$-integrable function $v : \clint{0,T} \function \H$ satisfying
\begin{alignat}{3}
  & y(t) \in \C(t) & \qquad &  \forall t \in \clint{0,T}, \label{y in C - BV - intro} \\
  & \! \D y = v \mu, & \label{Dy = w mu - intro} \\
  & v(t) + \partial I_{\C(t)}(y(t)) \ni 0 & \qquad & \text{for $\mu$-a.e. $t \in \clint{0,T}$}, \label{diff. incl. - BV - intro} \\
  & y(0) = \Proj_{\C(0)}(y_0), &  \label{in. cond. - BV - intro}
\end{alignat}
where $\D y$ denotes the distributional derivative of $y$. This in particular defines the solution operator 
$\M: \BV^{\r}(\clint{0,T};\Conv_\H) \function \BV^{\r}(\clint{0,T};\H)$ associating with $\C \in \BV^{\r}(\clint{0,T};\Conv_\H)$ the solution $y$ of \eqref{y in C - BV - intro}--\eqref{in. cond. - BV - intro}. Moreover if $\C$ is continuous then $\M(\C)$ is also continuous, and if $\C \in \AC(\clint{0,T};\Conv_\H)$ then $y = \M(\C) \in \W^{1,1}(0,T;\H)$ and $y$ satisfies 
\eqref{sweep constr intro}--\eqref{sweep in. cond. intro}. Usually one says that \eqref{diff. incl. - BV - intro} is the sweeping process driven by the moving convex set $\C$. For the theory of sweeping processes and some of their applications we also refer, e.g., to \cite{Mon93, ColGon99, Ben00, Mor02, ColMon03, Thi03, EdmThi06, BerVen10, BerVen11, BerVen15, DimMauSan16, Thi16} and their references.

A relevant feature of sweeping processes is their good behavior with respect to the change of time variable (cf. \cite[Proposition 2i]{Mor77}): if $\M$ is the solution operator of the sweeping process, associating with $\C$ the solution $y$ of \eqref{y in C - BV - intro}-\eqref{in. cond. - BV - intro}, then we have 
\begin{equation}\label{rate ind}
  \M(\C \circ \gamma) = \M(\C) \circ \gamma
\end{equation}
for every continuous increasing time-reparametrization $\gamma$. This property is also called 
\emph{rate independence}. For the theory of rate independent operators and systems we refer, e.g., to
\cite{KraPok89, BroSpr96, Kre97, Vis94, Mie05, MieRou15}. If 
$\C \in \BV(\clint{0,T};\Conv_\H) \cap \Czero(\clint{0,T};\Conv_\H)$, a natural reparametrization of time is given by the (normalized) arc length $\ell_\C : \clint{0,T} \function \clint{0,T}$ defined by
\[
  \ell_\C(t) := \frac{T}{\pV(\C,\clint{0,T})}\pV(\C, \clint{0,t}), \qquad t \in \clint{0,T},
\]
where $\pV(\C, \clint{0,t})$ denotes the variation of $\C$ over $\clint{0,t}$. Therefore there exists a Lipschitz continuous mapping $\Ctilde$ such that $\C = \Ctilde \circ \ell_\C$, thus the rate independence property yields 
\begin{equation}
  \M(\C) = \M(\Ctilde) \circ \ell_\C,
\end{equation} 
and $\M(\Ctilde) \in \Lip(\clint{0,T};\H)$, the space of Lipschitz functions. This reparametrization method was used by Moreau in \cite{Mor71, Mor74} in order to reduce  the sweeping process driven by an absolutely continuous moving set 
$\C(t)$ to the Lipschitz continuous case, while the reduction from 
$\BV(\clint{0,T};\Conv_\H) \cap \Czero(\clint{0,T};\Conv_\H)$ to $\Lip(\clint{0,T};\Conv_\H)$ is performed in 
\cite{Rec11c, Rec15b}.

Let us observe that if $u \in \W^{1,1}(\clint{0,T};\H)$ and $\C_u \in \AC^\r(\clint{0,T};\Conv_\H)$ is defined by 
$\C_u(t) := u(t) - \Z$, then we have $\P(u) = \M(\C_u)$. This remark naturally leads to the definition of the $\BV$-play operator $\P : \BV^{\r}(\clint{0,T};\H) \function \BV^{\r}(\clint{0,T};\H)$: $\P(u) := \M(\C_u)$ for any 
$u \in \BV^{\r}(\clint{0,T};\H)$. We can say that 
\emph{the play operator is a sweeping process driven by a moving convex set with constant shape}. There are other ways to define the play operator for $\BV$ inputs: we will provide a proof that $\P$ admits an integral variational formulation, to be more precise $y = \P(u)$ is the unique function such that \eqref{constraint} and \eqref{initial cond-intro} hold together with
\begin{equation}\label{play BV-integral inequality}
  \int_{\clint{0,T}} \duality{z(t) - u(t) + y(t)}{\de \D y(t)} \le 0 \qquad \forall z \in \BV^{\r}(\clint{0,T};\Z),
\end{equation}
where the integral is computed with respect to the Lebesgue-Stieltjes measure $\D y$. An analogous formulation is given in \cite{KreLau02} where the Young integral is used. Of course the play operator enjoys of the rate independence property which reads $\P(u \circ \gamma) = \P(u) \circ \gamma$ for every $u \in \BV^{\r}(\clint{0,T};\H)$ and every continuous increasing reparametrization $\gamma$ of time. In particular if 
$u \in \BV(\clint{0,T};\H) \cap \Czero(\clint{0,T};\H)$ and $\ell_u(t) := T\pV(u,\clint{0,t})/\pV(u,\clint{0,T})$, 
$t \in \clint{0,T}$, we have 
\begin{equation}\label{rate ind for P}
  \P(u) = \P(\utilde) \circ \ell_u,
\end{equation}
where $\utilde \in \Lip(\clint{0,T};\H)$ is such that $u = \utilde \circ \ell_u$.

The well-posedness of problem \eqref{var in-intro}--\eqref{initial cond-intro}, i.e. the continuity of the operator $\P$ with respect to various topologies, is a fundamental issue both from a theoretical and applicative point of view. The 
behavior of $\P : \BV^{\r}(\clint{0,T};\H) \function \BV^{\r}(\clint{0,T};\H)$ with respect to the topology of uniform convergence can be deduced, e.g.,  from the general results in \cite{Mor77} (cf. Thereom \ref{thm on P} below). The continuity of $\P$ with respect to the $\BV$ strict topology restricted to $\BV(\clint{0,T};\H) \cap \Czero(\clint{0,T};\H)$ was proved in \cite[Proposition 4.11, p. 46]{Kre97} if the boundary of $\Z$ satisfies suitable regularity conditions, and in \cite[Theorem 3.7]{Rec11a} for arbitrary $\Z$. In general  $\P$ is not $\BV$-strict continuous on the whole 
$\BV^{\r}(\clint{0,T};\H)$, it was proved in \cite{Rec11a} that the continuity in the strict topology holds if and only if 
$\Z = \{x \in \H\ :\ -\alpha \le \duality{f}{x} \le \beta\}$ for some $f \in \H \setmeno \{0\}$ and 
$\alpha, \beta \in \clint{0,\infty}$. In the one dimensional case it turns out $\P$ is always $\BV$-strict continuous on 
$\BV^{\r}(\clint{0,T};\ar)$ (see also \cite{Vis94, BroSpr96, Rec07, Rec09}).

In this paper we address the problem  of the continuity of $\P$ with respect to the classical $\BV$-norm topology induced
by the norm $\norm{u}{\BV} := \norm{u}{\infty} + \V(u,\clint{0,T})$, $u \in \BV^{\r}(\clint{0,T};\H)$. For absolutely continuous inputs the $\BV$-topology is exactly the standard $\W^{1,1}$-topology, and the continuity of the restriction of $\P$ to $\W^{1,1}(0,T;\H)$ was proved in \cite{Kre91} for finite dimensional $\H$ and in \cite{Kre97} for separable Hilbert spaces. For such spaces $\H$, the continuity of $\P$ in $\BV^{\r}(\clint{0,T};\H)$ (and in 
$\BV(\clint{0,T};\H) \cap \Czero(\clint{0,T};\H)$) is known only when $\Z$ has a smooth boundary (cf. 
\cite{BroKreSch04, KreRoc11}), in this case $\P$ is even locally Lipschitz continuous. Anyway this regularity assumption turns out be restrictive in many applications.

In the present paper we prove that $\P : \BV^{\r}(\clint{0,T};\H) \function \BV^{\r}(\clint{0,T};\H)$ is continuous with respect to the $\BV$-norm topology for \emph{every} arbitrary nonempty closed convex set $\Z$ (and with no separability 
assumptions on $\H$).

In order to describe what kind of difficulties arise in proving the general $\BV$-norm continuity of $\P$, let us briefly 
examine the known proofs in the more regular cases. 

If the input $u$ belongs to $\W^{1,1}(0,T;\H)$ and $x(t)$ solves \eqref{gradient flow for x - intro}, then $y'(t) = (\P(u))'(t)$ is a normal vector and $x'(t)$ is a tangential vector to $\Z$ at $x(t)$ in the sense of convex analysis. The proof given in \cite{Kre91} is strongly based on the orthogonal decomposition $u'(t) = x'(t) + y'(t)$. In the general $\BV$ case the distributional and the pointwise derivatives are different, so this procedure does not work.

If the input $u$ is an arbitrary $\BV$ function, but $\Z$ is smooth, then the proof provided in \cite{KreRoc11} relies upon an explicit formula for the (unique) unit normal vector to the boundary of $\Z$. If $\Z$ is not smooth there can be several 
unit normal vectors at a boundary point and this argument cannot be used.

These considerations, together with the rate independence property, suggest to try to use formula \eqref{rate ind for P}, at least for the continuous case, and somehow ``reduce'' the problem to the Lipschitz continuous case: indeed if 
$u \in \BV(\clint{0,T};\H) \cap \Czero(\clint{0,T};\H)$ then $\P(u) = \P(\utilde) \circ \ell_{u}$ and 
$\utilde, \P(\utilde) \in \Lip(\clint{0,T};\H)$, therefore one can try to get information on the $\BV$-norm continuity of 
$\P(u)$ by using the above orthogonal decomposition for the arc length reparametrization $\utilde$.

We are going to show that this procedure is actually possible, thus we are left with the discontinuous case and one can try to extend the previous reparametrization procedure. If $u \in \BV^{\r}(\clint{0,T};\H)$, then the reparametrization 
$\utilde$ is a Lipschitz function defined on the image $\ell_u(\clint{0,T})$, therefore we need to extend the definition of 
$\utilde$ to the whole $\clint{0,T}$, in other words we have to fill in the jumps of $u$. It is very natural to use segments, i.e. to define the Lipschitz continuous function $\utilde : \clint{0,T} \function \H$ to be affine on every interval
$\clint{\ell_u(t-),\ell_u(t)}$ and of course we still have $u = \utilde \circ \ell_u$. The length function $\ell_u$ is not continuous anymore, so rate independence does not apply, but anyway one may be tempted to use the formula 
$\P(\utilde) \circ \ell_u$. The issue here is that $\P(\utilde) \circ \ell_u \neq \P(u)$, as shown in \cite{Rec11a} (see \cite{KreRec14a,KreRec14b} for a detailed comparison of these two operators). We overcome this problem by considering the more general framework of sweeping processes: we consider the driving moving set $\C_u(t) = u(t) - \Z$ and we fill in the jumps of $\C_u$, (i.e. of $u$) with a suitable $\Conv_\H$-valued function, indeed using ``segments'' 
$(1-t)\A + t\B$ would produce the ``wrong'' output $\P(\utilde) \circ \ell_u$. The proper choice is connecting two sets $\A$ and $\B$ by geodesics of the form $\C(t) := (\A + D_{t\rho}) \cap (\B + D_{(1-t)\rho})$,  where $\rho$ is the Hausdorff distance between $\A$ and $\B$, and $D_r$ is the closed ball with center $0$ and radius $r$. Indeed in the paper \cite{Rec16} it is proved that if $\C \in \BV(\clint{0,T};\Conv_\H)$ and if $\Ctilde \in \Lip(\clint{0,T};\H)$ is the unique function such that $\C = \Ctilde \circ \ell_\C$ and 
\begin{equation}\label{convex geodesic}
  \Ctilde(\ell_\C(t-)(1-\lambda) + \ell_\C(t)\lambda)  =  
  (\C(t-) + D_{\lambda \rho_t}) \cap (\C(t) + D_{(1-\lambda)\rho_t}),
\end{equation}
for $t \in \clint{0,T}$, $\lambda \in \clint{0,1}$, with $\rho_t := \d_\hausd(\C(t-),\C(t))$, then $\M(\Ctilde) \circ \ell_\C$ is 
actually the solution of the sweeping process driven by $\C$, i.e. the formula $\M(\C) = \M(\Ctilde) \circ \ell_\C$ holds. 

In our particular situation if $u \in \BV^{\r}(\clint{0,T};\H)$ it follows that $\M(\Ctilde_{u}) \in \Lip(\clint{0,T};\H)$ is the play operator $\P(\utilde) \in \Lip(\clint{0,T};\H)$ on the set $\ell_u(\clint{0,T})$, where the pointwise derivative can be exploited, while outside of $\ell_u(\clint{0,T})$, on the ``jump set'', we can analyze $\P(u)$ by means of formula \eqref{convex geodesic}.  As a consequence, if $u_n \to u$ in $\BV^{\r}(\clint{0,T};\H)$ then the behaviour of the variation of $\P(u_n) = \M(\C_{u_n})$ can be studied with the help of the formula 
$\P(u_n) = \M(\C_{u_n}) = \M(\Ctilde_{u_n}) \circ \ell_{\C_{u_n}}$ and the $\BV$-norm continuity can be eventually proved by using tools from vector measure theory.

The paper is organized as follows. In the next section we present some preliminaries and in Section 
\ref{S:state main result} we state our main continuity result. The reparametrization technique for convex-valued functions is adapted to our situation in Section \ref{S:reparametrization} and it is exploited in Section \ref{S:integral representation}
to prove the integral representation for $\P$. In Section \ref{S:reduction} we reduce our problem to a Lipschitz continuous sweeping process. All the results of these sections are used in Section \ref{S:proof main thm} to prove the main theorem. 


\section{Preliminaries}\label{S:Preliminaries}

In this section we recall the main definitions and tools needed in the paper. {}{We denote by  $\en$  the set of natural numbers (without $0$).} If $\textsl{E}$ is a Banach space and $x, x_n \in E$ for every $n \in \en$, then the symbol 
$x_n \convergedeb x$ indicates that $x_n$ is weakly convergent to $x$ (cf., e.g., \cite{Bre83}). Given a subset $S$ of the real line $\ar$, if $\borel(S)$ denotes the family of Borel sets in $S$, $\mu : \borel(S) \function \clint{0,\infty}$ is a measure, $p \in \clint{1,\infty}$, then the space of $\E$-valued functions which are $p$-integrable with respect to $\mu$ will be denoted by $\L^p(S, \mu; \E)$ or simply by $\L^p(\mu; \E)$. We do not identify two functions which are equal 
$\mu$-almost everywhere ($\mu$-a.e.). For the theory of integration of vector valued functions we refer, e.g., to \cite[Chapter VI]{Lan93}. When $\mu = \leb^1$, the one dimensional Lebesgue measure, we write 
$\L^p(S; \E) := \L^p(S,\mu; \E)$. 


\subsection{Functions with values in a metric space}

In this subsection we assume that 
\begin{equation}\label{X complete metric space}
  \text{$(\X,\d)$ is a complete metric space},
\end{equation}
where we admit that $\d$ is an extended metric, i.e. $\X$ is a set and $\d : \X \times \X \function \clint{0,\infty}$ satisfies the usual axioms of a distance, but may take on the value $\infty$. The notion of completeness remains unchanged and the topology induced by $\d$ is defined in the usual way by means of the open balls 
$B_r(x_0) := \{x \in \X\ :\ \d(x,x_0) < r\}$ for $r > 0$ and $x_0 \in \X$, so that it satisfies the first axiom of countability. The general topological notions of interior, closure and boundary of a subset $\Y \subseteq \X$ will be respectively denoted by $\Int(\Y)$, $\Cl(\Y)$ and $\partial \Y$. If $x \in \X$ and $\Y \subseteq \X$, we also set 
$\d(x,\Y) := \inf_{y \in \Y} \d(x,y)$.
 
If $I \subseteq \ar$ is an interval and $f \in \X^I$ (the space of $\X$-valued functions defined on $I$), then $\cont(f)$
denotes the continuity set of $f$, and $\discont(f) := I \setmeno \cont(f)$. The set of $\X$-valued continuous functions defined on $I$ is denoted by $\Czero(I;\X)$. For $S \subseteq I$ we write 
$\Lipcost(f,S) := \sup\{\d(f(s),f(t))/|t-s|\ :\ s, t \in S,\ s \neq t\}$, $\Lipcost(f) := \Lipcost(f,I)$, the Lipschitz constant of $f$, 
and $\Lip(I;\X) := \{f \in \X^I\ :\ \Lipcost(f) < \infty\}$, the set of $\X$-valued Lipschitz continuous functions on $I$. We recall now the notion of $\BV$ function with values in a metric space (see, e.g., \cite{Amb90, Zie94}).

\begin{Def}
Given an interval $I \subseteq \ar$, a function $f \in \X^I$, and a subinterval $J \subseteq I$, the 
\emph{(pointwise) variation of $f$ on $J$} is defined by
\begin{equation}\notag
  \pV(f,J) := 
  \sup\left\{
           \sum_{j=1}^{m} \d(f(t_{j-1}),f(t_{j}))\ :\ m \in \en,\ t_{j} \in J\ \forall j,\ t_{0} < \cdots < t_{m} 
         \right\}.
\end{equation}
If $\pV(f,I) < \infty$ we say that \emph{$f$ is of bounded variation on $I$} and we set 
$\BV(I;\X) := \{f \in \X^I\ :\ \pV(f,I) < \infty\}$.
\end{Def}

It is well known that the completeness of $\X$ implies that every $f \in \BV(I;\X)$ admits one sided limits $f(t-), f(t+)$ 
at every point $t \in I$, with the convention that $f(\inf I-) := f(\inf I)$ if $\inf I \in I$, and that $f(\sup I+) := f(\sup I)$ if 
$\sup I \in I$. We set
\begin{align}
  & \BV^\l(I;\X) := \{f \in \BV(I;\X)\ :\ f(t-) = f(t) \quad \forall t \in I\}, \notag \\
  & \BV^\r(I;\X) := \{f \in \BV(I;\X)\ :\ f(t) = f(t+) \quad \forall t \in I\}. \notag
\end{align}
If $I$ is bounded we have $\Lip(I;\X) \subseteq \BV(I;\X)$.
\begin{Def}\label{D:def AC curve ags}
Assume that $p \in \clint{1,\infty}$. A mapping $f : I \function \X$ is called \emph{$\AC^{p}$-absolutely continuous} if there exists $m \in \L^{p}(I;\ar)$ such that
\begin{equation}
  \d(f(s),f(t)) \le \int_{s}^{t} m(\sigma) \de \sigma
  \qquad \forall s,t \in \clint{0,T},\ s \le t.
\end{equation}
The set of $\AC^{p}$-absolutely continuous functions is denoted by $\AC^{p}(I;\X)$.
\end{Def}

Clearly $\AC^{\infty}(I;\X) = \Lip(I;\X)$. If $I$ is bounded then $\AC^{p}(I;\X) \subseteq \BV(I;\X) \cap \Czero(I;\X)$ for every $p \in \clsxint{1,\infty}$, and $f \in \AC^1(I;\X)$ if and only if for every $\eps > 0$ there exists $\delta > 0$ such that $\sum_{j=1}^m \d(f(s_k),f(t_k)) < \eps$ whenever $m \in \en$ and $(\opint{s_k,t_k})_{k=1}^m$ is a family of mutually disjoint intervals with $\sum_{j=1}^m |t_k -s_k| < \delta$. In the next definition we recall two natural metrics in 
$\BV(I;\X)$.
\begin{Def}\label{metrics in BV}
For every $f, g \in \BV(I;\X)$ we set
\begin{align}
  & \d_{\infty}(f,g) := \sup_{t \in I} \d(f(t),g(t)), \\
  & \d_{us}(f,g) := \d_{\infty}(f,g) + |\pV(f,I) - \pV(g,I)|.
\end{align}
The metric $\d_{\infty}$ and $\d_{us}$ are called respectively \emph{uniform metric on $\BV(I;\X)$} and 
\emph{uniform strict metric on $\BV(I;\X)$}. We say that \emph{$u_n \to u$ uniformly strictly on $I$} if 
$\d_{us}(u_n,u) \to 0$ as $n \to \infty$. Let us recall that $\d_{us}$ is not complete and the topology induced by $\d_{us}$ is not linear if $\X$ is a Banach space.
\end{Def}

Now we recall the notion of geodesic.

\begin{Def}
Assume that $x,y \in \X$ and $\d(x,y) < \infty$. A function $g \in \Lip(\clint{0,1};\X)$ is called 
\emph{a geodesic connecting $x$ to $y$} if $g(0) = x$, $g(1) = y$, and $\pV(g,\clint{0,1}) = \d(x,y)$. 
\end{Def}


\subsection{Some convex analysis}

Let us assume that
\begin{equation}\label{H-prel}
\begin{cases}
  \text{$\H$ is a real Hilbert space with inner product $(x,y) \longmapsto \duality{x}{y}$,} \\
  \norm{x}{} := \duality{x}{x}^{1/2},
\end{cases}
\end{equation}
and we endow $\H$ with the natural metric defined by $\d(x,y) := \norm{x-y}{}$, $x, y \in \H$. We also use the notation 
\begin{equation}\notag
  D_r := \{x \in \H\ :\ \norm{x}{} \le r\}, \qquad r \ge 0,
\end{equation}
and we set
\begin{equation}\notag
  \Conv_\H := \{\K \subseteq \H\ :\ \K \ \text{nonempty, closed and convex} \}.
\end{equation}
If  $\K \in \Conv_\H$ and $x \in \H$, then $\Proj_{\K}(x)$ {}{denotes} the projection on $\K$, i.e. $y = \Proj_\K(x)$ is the unique point such that $d(x,\K) = \norm{x-y}{}$, {}{or equivalently   $y \in \K$ and $y$ satisfies the variational inequality
\begin{equation}\notag
  \duality{x - y}{v - y} \le 0 \qquad \forall v \in \K.
\end{equation}}
If $\K \in \Conv_\H$ and $x \in \K$, then $N_\K(x)$ denotes the \emph{(exterior) normal cone of $\K$ at $x$}:
\begin{equation}\label{normal cone}
  N_\K(x) := \{u \in \H\ :\ \duality{u}{v - x} \le 0\ \forall v \in \K\} = \Proj_\K^{-1}(x) - x.
\end{equation}
We endow the set $\Conv_{\H}$ with the Hausdorff distance. Here is the definition.

\begin{Def}
The \emph{Hausdorff distance} $\d_\hausd : \Conv_{\H} \times \Conv_{\H} \function \clint{0,\infty}$ is defined by
\begin{equation}\label{hausd dist}
  \d_{\hausd}(\A,\B) := 
  \max\left\{\sup_{a \in \A} \d(a,\B), \sup_{b \in \B} \d(b,\A)\right\}, \qquad
  \A, \B \in \Conv_{\H}.
\end{equation}
\end{Def}

Now we recall the notion of subdifferential (cf. \cite[Chapter 2]{Bre73}). If $\Psi : \H \function \clint{0,\infty}$ is convex 
and lower semicontinuous and $D(\Psi) := \{x \in \H\ :\ \Psi(x) \neq \infty\} \neq \void$, then for $x \in \H$ the \emph{subdifferential of $\phi$ at $x$} is defined by 
$\partial \Psi(x) := \{y \in \H\ :\ \duality{y}{v - x} + \Psi(x) \le \Psi(v)\ \forall v \in \H\}$. The \emph{domain} of $\partial \Psi$ 
is defined by $D(\partial \Psi) := \{x \in \H\ :\ \partial \Psi(x) \neq \void\}$. If $\mathcal{K} \in \Conv_{\H}$ and
$I_{\mathcal{K}}$, the \emph{indicator function of $\mathcal{K}$}, is defined by $I_{\mathcal{K}}(x) = 0$ if 
$x \in \mathcal{K}$ and $I_{\mathcal{K}}(x) = \infty$ if $x \not\in \mathcal{K}$, then 
$\partial I_{\mathcal{K}}(x) = N_{\mathcal{K}}(x)$ for every 
$x \in D(I_{\mathcal{K}}) = D(\partial I_{\mathcal{K}}) = \mathcal{K}$.


\subsection{Differential measures}

Let $\E$ be a Banach space with 
norm $\norm{\cdot}{\E}$ and let 
$I \subseteq \ar$ be 
an interval.
We recall that a \emph{(Borel) vector measure} on $I$ is a map $\mu : \borel(I) \function \E$ such that 
$\mu(\bigcup_{n=1}^{\infty} B_{n})$ $=$ $\sum_{n = 1}^{\infty} \mu(B_{n})$ whenever $(B_{n})$ is a sequence of mutually disjoint sets in $\borel(I)$.
The \emph{total variation of $\mu$} is the positive 
measure 
$\vartot{\mu} : \borel(I) \function \clint{0,\infty}$ defined by
\begin{align}\label{tot var measure}
  \vartot{\mu}(B)
  := \sup\left\{\sum_{n = 1}^{\infty} \norm{\mu(B_{n})}{\E}\ :\ 
                 B = \bigcup_{n=1}^{\infty} B_{n},\ B_{n} \in \borel(I),\ 
                 B_{h} \cap B_{k} = \varnothing \text{ if } h \neq k\right\}. \notag
\end{align}
The vector measure $\mu$ is said to be \emph{with bounded variation} if $\vartot{\mu}(I) < \infty$. In this case the equality $\norm{\mu}{} := \vartot{\mu}(I)$ defines a norm on the space of measures with bounded variation (see, e.g. \cite[Chapter I, Section  3]{Din67}). 

If $\nu : \borel(I) \function \clint{0,\infty}$ is a positive bounded Borel measure and if $g \in \L^1(I,\nu;\E)$, then $g\nu$ will denote the vector measure defined by $g\nu(B) := \int_B g\de \nu$ for every $B \in \borel(I)$. In this case 
$\vartot{g\nu}(B) = \int_B \norm{g(t)}{E}\de \nu$ for every $B \in \borel(I)$ (see \cite[Proposition 10, p. 174]{Din67}).

Let $\E_{j}$, $j = 1, 2, 3$, be Banach spaces with norms $\norm{\cdot}{\E_{j}}$ and let 
$\E_{1} \times \E_{2} \function \E_{3} : (x_{1},x_{2}) \longmapsto x_{1} \bullet x_{2}$ be a bilinear form such that 
$\norm{x_{1} \bullet x_{2}}{\E_{3}} \le \norm{x_{1}}{\E_{1}} \norm{x_{2}}{\E_{2}}$ for every $x_{j} \in \E_{j}$, $j=1, 2$. Assume that $\mu : \borel(I) \function \E_{2}$ is a vector measure with bounded variation and let $f : I \function \E_{1}$ be a \emph{step map with respect to $\mu$}, i.e. there exist $f_{1}, \ldots, f_{m} \in \E_{1}$ and 
$A_{1}, \ldots, A_{m} \in \borel(I)$ mutually disjoint such that $\vartot{\mu}(A_{j}) < \infty$ for every $j$ and 
$f = \sum_{j=1}^{m} \indicator_{A_{j}} f_{j},$ where $\indicator_{S} $ is the characteristic function of a set $S$, i.e. 
$\indicator_{S}(x) := 1$ if $x \in S$ and $\indicator_{S}(x) := 0$ if $x \not\in S$. For such $f$ we define 
$\int_{I} f \bullet \de \mu := \sum_{j=1}^{m} f_{j} \bullet \mu(A_{j}) \in \E_{3}$. If $\Step(\vartot{\mu};\E_{1})$ is the set of 
$\E_1$-valued maps with respect to $\mu$, then the map 
$\Step(\vartot{\mu};\E_{1})$ $\function$ $\E_{3} : f \longmapsto \int_{I} f \bullet \de \mu$ is linear and continuous when 
$\Step(\vartot{\mu};\E_{1})$ is endowed with the $\L^{1}$-semimetric 
$\norm{f - g}{\L^{1}(\vartot{\mu};\E_{1})} := \int_I \norm{f - g}{\E_{1}} \de \vartot{\mu}$. Therefore it admits a unique continuous extension $\mathsf{I}_{\mu} : \L^{1}(\vartot{\mu};\E_{1}) \function \E_{3}$ and we set 
\[
  \int_{I} f \bullet \de \mu := \mathsf{I}_{\mu}(f), \qquad f \in \L^{1}(\vartot{\mu};\E_{1}).
\]
We will use the previous integral in two particular cases, namely when 
\begin{itemize}
\item[a)] 
  $\E_{1} = \ar$, $\E_{2} = \E_{3} = \H$, $\lambda \bullet x := \lambda x$ 
  ($\int_{I} f \bullet \de \mu = \int_{I} f \de \mu$, integral of a real function with respect to a vector measure); 
\item[b)]
  $\E_{1} = \E_{2} = \H$, $\E_{3} = \ar$, $x_{1} \bullet x_{2} := \duality{x_{1}}{x_{2}}$ 
  ($\int_{I} f \bullet \de \mu = \int_{I} \duality{f}{\de \mu}$, integral of a vector function with respect to a vector measure).
\end{itemize}
In the situation (b) with $\mu = g\nu$, $\nu$ bounded positive measure and $g \in \L^{1}(\nu;\H)$, arguing first on step functions, and then taking limits, it is easy to check that $\int_I\duality{f}{\de(g\nu)} = \int_I \duality{f}{g}\de \nu$ for every $f \in \L^{\infty}(\mu;\H)$. The following results (cf., e.g., \cite[Section III.17.2-3, pp. 358-362]{Din67}) provide a connection between functions with bounded variation and vector measures. 

\begin{Thm}\label{existence of Stietjes measure}
For every $f \in \BV(I;\H)$ there exists a unique vector measure of bounded variation $\mu_{f} : \borel(I) \function \H$ such that 
\begin{align}
  \mu_{f}(\opint{c,d}) = f(d-) - f(c+), \qquad \mu_{f}(\clint{c,d}) = f(d+) - f(c-), \notag \\ 
  \mu_{f}(\clsxint{c,d}) = f(d-) - f(c-), \qquad \mu_{f}(\cldxint{c,d}) = f(d+) - f(c+). \notag 
\end{align}
whenever $c < d$ and the left hand side of each equality makes sense. 

Vice versa if $\mu : \borel(I) \function \H$ is a vector measure with bounded variation, and if $f_{\mu} : I \function \H$ is defined by $f_{\mu}(t) := \mu(\clsxint{\inf I,t} \cap I)$, then $f_{\mu} \in \BV(I;\H)$ and $\mu_{f_{\mu}} = \mu$.
\end{Thm}

\begin{Prop}
Let $f  \in \BV(I;\H)$, let $g : I \function \H$ be defined by $g(t) := f(t-)$, for $t \in \Int(I)$, and by $g(t) := f(t)$, if 
$t \in \partial I$, and let $V_{g} : I \function \ar$ be defined by $V_{g}(t) := \V(g, \clint{\inf I,t} \cap I)$. Then  
$\mu_{g} = \mu_{f}$ and $\vartot{\mu_{f}} = \mu_{V_{g}} = \V(g,I)$.
\end{Prop}

The measure $\mu_{f}$ is called \emph{Lebesgue-Stieltjes measure} or \emph{differential measure} of $f$. Let us see the connection with the distributional derivative. If $f \in \BV(I;\H)$ and if $\overline{f}: \ar \function \H$ is defined by
\begin{equation}\label{extension to R}
  \overline{f}(t) :=
  \begin{cases}
    f(t) 	& \text{if $t \in I$} \\
    f(\inf I)	& \text{if $\inf I \in \ar$, $t \not\in I$, $t \le \inf I$} \\
    f(\sup I)	& \text{if $\sup I \in \ar$, $t \not\in I$, $t \ge \sup I$}
  \end{cases},
\end{equation}
then, as in the scalar case, it turns out (cf. \cite[Section 2]{Rec11a}) that $\mu_{f}(B) = \D \overline{f}(B)$ for every 
$B \in \borel(\ar)$, where $\D\overline{f}$ is the distributional derivative of $\overline{f}$, i.e.
\[
  - \int_\ar \varphi'(t) \overline{f}(t) \de t = \int_{\ar} \varphi \de \D \overline{f} 
  \qquad \forall \varphi \in \Czero_{c}^{1}(\ar;\ar),
\]
$\Czero_{c}^{1}(\ar;\ar)$ being the space of real continuously differentiable functions on $\ar$ with compact support.
Observe that $\D \overline{f}$ is concentrated on $I$: $\D \overline{f}(B) = \mu_f(B \cap I)$ for every $B \in \borel(I)$, hence in the remainder of the paper, if $f \in \BV(I,\H)$ then we will simply write
\begin{equation}
  \D f := \D\overline{f} = \mu_f, \qquad f \in \BV(I;\H),
\end{equation}
and from the previous discussion it follows that 
\begin{equation}
  \norm{\D f}{} = \vartot{\D f}(\clint{0,T}) = \norm{\mu_f}{}  \qquad \forall f \in \BV(I;\H).
\end{equation}
If $I$ is bounded and $p \in \clint{1,\infty}$, then the classical Sobolev space $\W^{1,p}(I;\H)$ consists of those functions $f \in \Czero(I;\H)$ such that $\D f = g\leb^1$ for some $g \in \L^p(I;\H)$ and we have  $\W^{1,p}(I;\H) = \AC^p(I;\H)$. Let us also recall that if $f \in \W^{1,1}(I;\H)$ then {}{ the derivative $f'(t)$  exists} for $\leb^1$-a.e. in $t \in I$, 
$\D f = f' \leb^1$, and $\V(f,I) = \int_I\norm{f'(t)}{}\de t$ (cf., e.g. \cite[Appendix]{Bre73}).

In \cite[Lemma 6.4 and Theorem 6.1]{Rec11a} it is proved that

\begin{Prop}\label{P:BV chain rule}
Assume that $J \subseteq \ar$ is a bounded interval and $h : I \function J$ is nondecreasing.
\begin{itemize}
\item[(i)]
  $\D h(h^{-1}(B)) = \leb^{1}(B)$ for every $B \in \borel(h(\cont(h)))$.
\item[(ii)]
 If $f \in \Lip(J;\H)$ and $g : I \function \H$ is defined by
\begin{equation}\notag
  g(t) := 
  \begin{cases}
    f'(h(t)) & \text{if $t \in \cont(h)$} \\
    \ \\
    \dfrac{f(h(t+)) - f(h(t-))}{h(t+) - h(t-)} & \text{if $t \in \discont(h)$}
  \end{cases},
\end{equation}
then $f \circ h \in \BV(I;\H)$ and $\D\ \!(f \circ h) = g \D h$. This result holds even if $f'$ is replaced by any of its 
$\leb^{1}$-representatives.
\end{itemize}
\end{Prop}

In the remainder of the paper we address our attention to left and right continuous functions of bounded variation on a compact interval $\clint{a,b}$, ($-\infty < a < b < \infty$). In this case we have
\begin{equation}\label{BV-norm-left}
  \norm{\D f}{} = \norm{\mu_f}{} = \pV(f,\opint{a,b}) + \norm{f(a+) - f(a)}{} = \pV(f,\clint{a,b})
  \qquad \forall f \in \BV^\l(\clint{a,b};\H),
\end{equation}
\begin{equation}\label{BV-norm-right}
  \norm{\D f}{} = \norm{\mu_f}{} = \pV(f,\opint{a,b}) + \norm{f(b) - f(b-)}{} = \pV(f,\clint{a,b})
  \qquad \forall f \in \BV^\r(\clint{a,b};\H),
\end{equation}
therefore when we consider left (resp. right) continuous functions we are essentially dealing with Lebesgue equivalence classes of functions with a special view on the initial point $a$ (resp. final point $b$), allowing us to take into account Dirac masses at $a$ or $b$. We will consider on $\BV(\clint{a,b};\H)$ the classical $\BV$-norm defined by
\begin{equation}\label{def BVnorm}
  \norm{f}{\BV} := \norm{f}{\infty} + \V(f,\clint{a,b}), \qquad f \in \BV(\clint{a,b};\H),
\end{equation}
which is equivalent to the norm defined by $\interleave f \interleave_{\BV} := \norm{f(0)}{} + \V(f,\clint{a,b})$, 
$f \in \BV(\clint{a,b};\H)$. Observe that we have 
\[
\norm{f}{\BV} = \norm{f}{\infty} + \norm{\D f}{} \qquad \forall f \in \BV^\l(\clint{a,b};\H) \cup \BV^\r(\clint{a,b};\H).
\]
The topology induced by $\d_{us}$ is clearly weaker than the one induced by $\norm{\cdot}{\BV}$.


\section{Statement of the main result}\label{S:state main result}

In this section we state the main theorem of the present paper. To this aim we recall the well known existence results about sweeping processes and the play operator.

We assume that
\begin{gather}
   \Z \in \Conv_\H, \label{Z closed convex-results}  \\
   0 < T < \infty. \label{T>0-results}
\end{gather}
Let us start with the general existence result for sweeping processes proved in \cite{Mor77}.

\begin{Thm}\label{T:existence general BVsweep}
If $\C \in \BV^\r(\clint{0,T};\Conv_\H)$ and $y_0 \in \H$, then there is a unique $\M(y_0,\C) := y \in \BV^\r(\clint{0,T};\H)$, such that there exist a measure $\mu : \borel(\clint{0,T}) \function \clsxint{0,\infty}$ and a function $v \in \L^1(\mu;\H)$ satisfying 
\begin{alignat}{3}
  & y(t) \in \C(t) & \qquad & \forall t \in \clint{0,T}, \label{y in C - BVsweep} \\
  & \!\D y = v \mu, \label{Dy = w mu - BVsweep} \\
  & v(t) + \partial I_{\C(t)}(y(t)) \ni 0 & \qquad & \text{for $\mu$-a.e. $t \in \clint{0,T}$}, \label{diff. incl. - BVsweep} \\
  & y(0) = \Proj_{\C(0)}(y_{0}). \label{in. cond. - BVsweep}
\end{alignat}
The resulting solution operator  $\M : \H \times \BV^\r(\clint{0,T};\Conv_\H)$ $\function$ $\BV^\r(\clint{0,T};\H)$ is continuous when $\BV^\r(\clint{0,T};\Conv_\H)$ is endowed with the topology induced by $\d_{us}$ and 
$\BV^\r(\clint{0,T};\H)$ is endowed with the topology induced by $\d_\infty$. We have
$\M(\H \times \BV(\clint{0,T};\Conv_\H) \cap \Czero(\clint{0,T};\Conv_\H)) \subseteq 
\BV(\clint{0,T};\H) \cap \Czero(\clint{0,T};\H)$. For every $p \in \clint{1,\infty}$ we have
$\M(\H \times \AC^p(\clint{0,T};\Conv_\H)) \subseteq \W^{1,p}(0,T;\H)$, and if
$\C$ $\in$ $\AC^p(\clint{0,T};\Conv_\H)$ then $y = \M({y_0},\C)$ is the unique function satisfying  
\eqref{y in C - BVsweep}, \eqref{in. cond. - BVsweep}, and
\begin{equation}
   y'(t) + \partial I_{\C(t)}(y(t)) \ni 0 \qquad \text{for $\leb^{1}$-a.e. $t \in \clint{0,T}$}.\label{diff. incl. - Lip}
\end{equation}
In this case it holds 
\begin{equation}\label{estimate Lip sweep proc}
  \Lipcost(y) \le \Lipcost(\C).
\end{equation}
\end{Thm}

\begin{Rem}
The uniqueness and existence statements of the previous theorem are provided in 
\cite[Proposition 3a, Proposition 3b]{Mor77}. The continuity of $\M$ is proved in \cite[Proposition 2g]{Mor77}. The regularity statements and \eqref{estimate Lip sweep proc} are proved in \cite[Corollary 2c]{Mor77} (for 
$p \in \opint{1,\infty}$ see \cite[Proposition 3.2]{Rec11c}), while \eqref{diff. incl. - Lip} is shown in 
\cite[Proposition 3c]{Mor77}.
\end{Rem}

The differential inclusion \eqref{diff. incl. - BVsweep} is usually called 
\emph{sweeping process driven by the moving convex set $\C$}. Now we recall the definition of the 
\emph{play operator}, that is the operator solution of the sweeping process driven by a moving convex set with constant shape.

\begin{Def}
For any $u \in \BV^\r(\clint{0,T};\H)$ define $\C_u \in \BV^\r(\clint{0,T};\Conv_\H)$ by
\begin{equation}\label{def Cu}
  \C_u(t) := u(t) - \Z, \qquad t \in \clint{0,T}.
\end{equation} 
The operator $\P : \Z \times \BV^\r(\clint{0,T};\H) \function \BV^\r(\clint{0,T};\H)$ defined by
\begin{equation}\label{def. play}
  \P(z_0, u) := \M({u(0)-z_0},\C_u), \qquad z_0 \in \Z, \ u \in \BV^\r(\clint{0,T};\H),
\end{equation} 
is called \emph{play operator (with characteristic $\Z$).}
\end{Def}

The following theorem will be proved in Section \ref{S:integral formulation} and summarizes the main properties of $\P$ inherited by Theorem \ref{T:existence general BVsweep}. It also provides an integral variational characterization of $\P$.

\begin{Thm}\label{thm on P}
The operator $\P : \Z \times \BV^\r(\clint{0,T};\H) \function \BV^\r(\clint{0,T};\H)$ is continuous if $\BV^\r(\clint{0,T};\H)$ is endowed with the topology induced by $\d_{us}$ in the domain and by $\d_\infty$ in the codomain. We have 
$\P(\Z \times \BV(\clint{0,T};\H) \cap \Czero(\clint{0,T};\H))$ $\subseteq$ $\BV(\clint{0,T};\H) \cap \Czero(\clint{0,T};\H)$. 
Moreover if $z_0 \in \Z$ and $u \in \BV^\r(\clint{0,T};\H)$, then $\P(z_0, u) = y \in \BV^\r(\clint{0,T};\H)$ is the unique function such that
\begin{align}
  & u(t) - y(t) \in \Z \qquad \text{$\forall t \in \clint{0,T}$}, \label{u-y in Z-KrLa} \\
  & \int_{\clint{0,T}} \bduality{z(t) - u(t) + y(t)}{\de \D y(t)} \le 0 \qquad \forall z \in \BV^\r(\clint{0,T};\Z), 
       \label{var in KrLa} \\
  & u(0) - y(0) = z_{0}. \label{initial cond for P KrLa}		 
\end{align}
Equivalently $\P(z_0,u) = y \in \BV^\r(\clint{0,T};\H)$ is the unique function satisfying \eqref{u-y in Z-KrLa}, 
\eqref{initial cond for P KrLa}, and
\begin{equation}\label{var in KrLa-test Linfty}
  \int_{\clint{0,T}} \bduality{z(t) - u(t) + y(t)}{\de \D y(t)} \le 0 
  \qquad \forall z \in \L^\infty(\clint{0,T};\H),\ z(\clint{0,T}) \subseteq \Z.
\end{equation}
For
every $p \in \clint{1,\infty}$ we have $\P(\Z \times \W^{1,p}(0,T;\H))$ $\subseteq$ $\W^{1,p}(0,T;\H)$ and if $u \in \W^{1,p}(0,T;\H)$ then $\P(z_0, u) = y$ is the unique function satisfying \eqref{u-y in Z-KrLa}, \eqref{initial cond for P KrLa}, and
\begin{equation}
  \lduality{z(t) - u(t) - y(t)}{y'(t)} \le 0 \qquad 
  \text{for $\leb^1$-a.e. $t \in \clint{0,T},\ \forall z \in \Z$}. \label{inequality for x} 
\end{equation}
\end{Thm}

\begin{Rem}
The integral formulation \eqref{var in KrLa} (or \eqref{var in KrLa-test Linfty}) is analogous to the formulation used in
\cite{KreLau02} where the Young integral is used. Our proof in Section \ref{S:integral formulation} is completely independent and uses only tools from differentiation theory.
\end{Rem}

Here is the main theorem that we will prove in Section \ref{S:proof main thm}.

\begin{Thm}\label{T:main thm}
The play operator $\P : \Z \times \BV^\r(\clint{0,T};\H) \function \BV^\r(\clint{0,T};\H)$ is continuous if 
$\BV^\r(\clint{0,T};\H)$ is endowed with the topology induced by the $\BV$-norm \eqref{def BVnorm}. 
\end{Thm}

\begin{Rem}If we deal with left continuous functions the formulations of our problem need to be modified accordingly. More precisely if $\C \in \BV^\l(\clint{0,T};\Conv_\H)$ and $y_0 \in \H$ then there is a unique 
$\M(y_0,\C) := y \in \BV^\l(\clint{0,T};\H)$ such that there exist a measure 
$\mu : \borel(\clint{0,T}) \function \clsxint{0,\infty}$ and a function $v \in \L^1(\mu;\H)$ satisfying \eqref{y in C - BVsweep}, \eqref{Dy = w mu - BVsweep}, \eqref{in. cond. - BVsweep} and
\[
  v(t) + \partial I_{\C(t+)}(y(t+)) \ni 0 \qquad  \text{for $\mu$-a.e. $t \in \clint{0,T}$}.
\]
This can be easily proved by adapting the proof of \cite{Mor77} to the left continuous case, or one can argue by reducing to Lipschitz inputs by using exacly the same argument as in  \cite[Theorem 6.1]{Rec16}. The play operator 
$\P : \Z \times \BV^\l(\clint{0,T};\H) \function \BV^\l(\clint{0,T};\H)$ is defined by $\P(z_0, u) := \M({u(0)-z_0},\C_u)$ for $z_0 \in \Z$, $u \in \BV^\l(\clint{0,T};\H)$, where $\C_u \in \BV^\l(\clint{0,T};\Conv_\H)$ is given by $\C_u(t) := u(t) - \Z$, 
$t \in \clint{0,T}$. Moreover $\P(z_0, u) = y \in \BV^\l(\clint{0,T};\H)$ is the unique function satisfying \eqref{u-y in Z-KrLa}, \eqref{initial cond for P KrLa}, and
\[
  \int_{\clint{0,T}} \bduality{z(t) - u(t+) + y(t+)}{\de \D y(t)} \le 0 \qquad \forall z \in \BV^\l(\clint{0,T};\Z).
\]
A motivation of the use of left continuous functions is, e.g., the fact that the viscous regularizations of rate independent processes converge to a left continuous function for the viscosity coefficient approaching zero (cf. 
\cite[Theorem 2.4]{KreLie09}). Modifying our proofs in the obvious way we get the following

\begin{Thm}\label{T:main thm left version}
The play operator $\P : \Z \times \BV^\l(\clint{0,T};\H) \function \BV^\l(\clint{0,T};\H)$ is continuous if 
$\BV^\l(\clint{0,T};\H)$ is endowed with the topology induced by the $\BV$-norm \eqref{def BVnorm}. 
\end{Thm}
\end{Rem}


\section{Reparametrizations}\label{S:reparametrization}

In this section we recall the notion of a reparametrization by the arclength of a $\Conv_\H$-valued $\BV$-function  introduced in \cite{Rec16}. This will be the key tool for the proof of our main theorem. We start with a more general notion of reparametrization in a general metric space setting.

Assume that \eqref{X complete metric space} holds and set
\begin{equation}\label{pairs of X with finite distance}
  \Phi_X := \{(x,y) \in \X \times \X\ :\ 0 < d(x,y) < \infty\}. \notag 
\end{equation}


\subsection{Reparametrization associated to a family of geodesics}

Let us recall \cite[Proposition 5.1]{Rec16}.

\begin{Prop}\label{P:reparametrization}

For $f \in \BV^\r(\clint{0,T};\X)${}{,} let $\ell_f : \clint{0,T} \function \clint{0,T}$ be defined by
\begin{equation}\notag
  \ell_f(t) := 
  \begin{cases}
    \dfrac{T}{\V(f,\clint{0,T})}\pV(f,\clint{0,t}) & \text{if $\V(f,\clint{0,T}) \neq 0$} 	\\
    															\\
    0 								& \text{if $\V(f,\clint{0,T}) = 0$}
  \end{cases}
    \qquad t \in \clint{0,T}.
\end{equation}
\begin{itemize}
\item[(i)]
  We have that $\ell_f$ is nondecreasing, $\discont(f) = \discont(\ell_f)$, and
  \begin{equation}\label{ell(0,T)=}
    \ell_f(\clint{0,T}) = \clint{0,T} \setmeno \bigcup_{t \in \discont(f)} \cldxint{\ell_f(t-),\ell_f(t)}.
  \end{equation}
Moreover there is a unique $F : \ell_f(\clint{0,T}) \function \X$ such that
  \begin{align}
    f = F \circ \ell_f, \qquad \Lipcost(F) \le \frac{\V(f,\clint{0,T})}{T}. \notag
  \end{align}
\item[(ii)]
Let $\mathscr{G} = (g_{(x,y)})_{(x,y) \in \Phi}$ be a family of geodesics connecting $x$ to $y$ for every $(x,y) \in \Phi_X$. If  $\ftilde : \clint{0,T} \function \X$ is defined by
\begin{equation}\label{reparametrization w.r.t. F}
  \ftilde(\sigma) := 
  \begin{cases}
    F(\sigma) & \text{if $\sigma \in \ell_f(\clint{0,T})$} 	\\
     										\\
    g_{(f(t-),f(t))}\left(\dfrac{\sigma - \ell_f(t-)}{\ell_f(t) - \ell_f(t-)}\right) 
       & \text{if $\sigma \in \cldxint{\ell_f(t-),\ell_f(t)}$, $t \in \discont(f)$}
  \end{cases},
\end{equation}
then 
\begin{gather}
   f = \ftilde \circ \ell_f, \label{u = utilde o ell} \\
   \pV(\ftilde,\clint{0,T}) = \pV(f,\clint{0,T}), \\
   \Lipcost(\ftilde) = \Lipcost(F) \le \frac{\V(f,\clint{0,T})}{T}, \label{Lip(ftilde) = Lip(F)}
\end{gather}
and
\begin{equation}\label{image of utilde}
  \ftilde(\clint{0,T}) 
    = f(\clint{0,T}) \bigcup \left( \bigcup_{t \in \discont(f)} g_{(f(t-),f(t))}([0,1]) \right). \notag
\end{equation}
\end{itemize}
\end{Prop}

\subsection{The Hilbert case}

Let us consider Proposition \ref{P:reparametrization} with $X = \H$. In this case the family 
$\mathscr{G} = (g_{(x,y)})_{(x,y) \in \Phi_\H}$ is defined a fortiori by $g_{(x,y)}(t) := (1-t)x + ty$, $t \in \clint{0,1}$. Therefore for every $u \in \BV^\r(\clint{0,T};\H)$ there exists a unique $\utilde \in \Lip(\clint{0,T};\H)$ such that 
$\Lipcost(\utilde,\clint{0,T}) \le \V(u,\clint{0,T})/T$ and 
\begin{align}
  & u = \utilde \circ \ell_u \label{reparametrization1} \\ 
  & \utilde(\ell_u(t-)(1-\lambda) + \ell_u(t)\lambda) =  
    (1-\lambda)u(t-) + \lambda u(t) 
     \qquad \forall t \in \clint{0,T},\ \forall \lambda \in \clint{0,1}. \label{reparametrization2}
\end{align}

Moreover from \cite[Lemma 4.3, Proposition 4.10]{Rec11a} we immediately infer the following

\begin{Prop}\label{P:vntilde -> vtilde}
For $u \in \BV^\r(\clint{0,T};\H)$ we have that
\begin{equation}\label{|utilde'| costante}
  \norm{\utilde'(\sigma)}{\H} = \frac{\V(u,\clint{0,T})}{T}   
  \qquad \text{for $\leb^1$-a.e. $\sigma \in \clint{0,T}$}.
\end{equation}
If $u_n \in \BV^\r(\clint{0,T};\H)$ for every $n \in \en$ and $\d_{us}(u_n,u) \to 0$ then $\utilde_n \to \utilde$ in 
$\W^{1,p}(0,T;\H)$ for every $p \in \clsxint{1,\infty}$.
\end{Prop}


\subsection{Reparametrization of ``convex-valued'' curves}\label{S:integral formulation}

Let us {}{consider } the situation of Proposition \ref{P:reparametrization} in the case when $X = \Conv_\H$, and the family 
$\mathscr{G} = \left( \G_{(\A,\B)}\right)$ is defined by
\begin{equation}\label{catching-up geodesic-2}
  \G_{(\A,\B)}(t) := (\A + D_{t\rho}) \cap (\B + D_{(1-t)\rho}), \quad t \in \clint{0,1}, \quad 
  \rho := \d_\hausd(\A,\B) < \infty.
\end{equation}
The mapping $\G_{(\A,\B)}$ is actually a geodesic in $\Conv_\H$ (cf. \cite[Proposition 4.1]{Rec16} and 
\cite[Theorem 1]{Ser98}), therefore if $\C \in \BV^\r(\clint{0,T};\Conv_\H)$, then {}{there} exists a unique 
$\Ctilde \in \Lip(\clint{0,T};\Conv_\H)$ such that $\Lipcost(\Ctilde,\clint{0,T})$ $\le$ $\V(\C,\clint{0,T})/T$ and 
\begin{align}
  & \C = \Ctilde \circ \ell_\C \label{reparametrization1-Conv} \\ 
  & \Ctilde(\ell_\C(t-)(1-\lambda) + \ell_\C(t)\lambda)  =  
     (\C(t-) + D_{\lambda \rho_t}) \cap (\C(t) + D_{(1-\lambda)\rho_t}) \notag \\
  & \phantom{\ \ \ \ \ \ \ \ \ \ \ \ \ \ \ \ }
     \qquad \forall t \in \clint{0,T},\ \forall \lambda \in \clint{0,1}, \text{ with $\rho_t := \d_\hausd(\C(t-),\C(t))$}. \label{reparametrization2-Conv} 
\end{align}

Moreover the following Proposition is proved in \cite[Corollary 5.1]{Rec16}.

\begin{Prop}\label{Cn -> C => Ctilden -> Ctilde}
If $\C, \C_n \in \BV^\r(\clint{0,T};\Conv_\H)$ for every $n \in \en$ and $\d_{us}(\C_n,\C) \to 0$ as $n \to \infty$, then 
$\d_{us}(\Ctilde_n,\Ctilde) \to 0$ as $n \to \infty$.
\end{Prop}

The family of geodesics (\ref{catching-up geodesic-2}) is studied in \cite{Rec16} in connection with sweeping processes. Indeed in \cite[Theorem 6.1]{Rec16} the following result is proved.

\begin{Thm}\label{T:Lip-reduction of BVsweep}
If $y_0 \in \H$ then $\M(y_0,\C) = \M(y_0, \Ctilde) \circ \ell_\C$ for every $\C \in \BV^\r(\clint{0,T};\Conv_\H)$.
\end{Thm}

The previous Theorem \ref{T:Lip-reduction of BVsweep} allows us to reduce any $\BV$-sweeping process to a Lipschitz continuous one. In order to study $\M(y_0,\Ctilde)$ we need the following Lemma proved in \cite[Lemma 4.4]{Rec16}.

\begin{Lem}\label{L:particular sweeping process}
Let $\A, \B \in \Conv_\H$ be such that $\d_\hausd(\A,\B) < \infty$ and let $\G_{(\A, \B)} : \clint{0,1} \function \Conv_\H$ be defined by \eqref{catching-up geodesic-2}. For $u_0 \in \A$ let $t_0 \in \clint{0,1}$ be the unique number such that 
$\norm{u_0 - \Proj_\B(u_0)}{} = (1-t_0)\d_\hausd(\A,\B)$. Then
\begin{equation}\label{solution for catching-up geodesic-2}
  \M(u_0, \G_{(\A, \B)})(t) :=
  \begin{cases}
    u_0 & \text{if $t \in \clsxint{0,t_0}$} \\
    u_0 + \dfrac{t-t_0}{1-t_0}(\Proj_{\B}(u_0) - u_0) & \text{if $t \in \clsxint{t_0,1}$} \\
    \Proj_\B(u_0) & \text{if $t = 1$}
  \end{cases}.
\end{equation}
\end{Lem}

\section{Integral representation for $\P$}\label{S:integral representation}

The reparametrization by the arc length allows to give a simple 

\begin{proof}[Proof of Theorem \ref{thm on P}]
We only have to prove the statements about the integral formulations of $\P$, the remaining assertions following from Theorem \ref{T:existence general BVsweep}. Assume that $y = \P(z_0, u)$, then \eqref{u-y in Z-KrLa}, 
\eqref{initial cond for P KrLa} hold, and there exist a measure $\mu : \borel(\clint{0,T}) \function \clsxint{0,\infty}$ and a function $v \in \L^1(\mu;\H)$ such that $\D y = v \mu$. If $z \in \L^\infty(\mu;\H)$ and $z(\clint{0,T}) \subseteq \Z$ then from \eqref{diff. incl. - BVsweep} it follows that $\duality{z(t) - u(t) + y(t)}{v(t)} \le 0$ for $\mu$-a.e. $t \in \clint{0,T}$. Thus integrating this inequality with respect to $\mu$ we find 
\[
  0 \ge \int_{\clint{0,T}} \duality{z(t) - u(t) + y(t)}{v(t)} \de \mu = \int_{\clint{0,T}} \duality{z(t) - u(t) + y(t)}{\de \D y(t)},
\]
thus \eqref{var in KrLa-test Linfty} and \eqref{var in KrLa} hold. Vice versa let us assume that, $y \in \BV^\r(\clint{0,T};\H)$ satisfies \eqref{u-y in Z-KrLa}--\eqref{initial cond for P KrLa}. Since $y = \widetilde{y} \circ \ell_y$, from Proposition \ref{P:BV chain rule} we get that $\D y = v \D \ell_y,$ where $v : \clint{0,T} \function \H$ is defined by 
$v(t) := \widetilde{y}\ \!'(\ell_y(t))$ for $t \in \cont(y)$ and 
$v(t) := (\widetilde{y}(\ell_y(t)) - \widetilde{y}(\ell_y(t-)))/(\ell_y(t) - \ell_y(t-))$ for $t \in \discont(y)$. Now set 
$C := \{s \in \cont(\ell_y)\ :\ \D\ell_y(\opint{s-h,s+h} \cap \clint{0,T}) \neq 0\ \forall h > 0)\}$ (i.e. $C$ is the set of continuity points of $\ell_y$ which do not lie in the interior of a constancy interval of $\ell_y$) and observe that 
$\lim_{h \searrow 0} \D\ell_y(\opint{s-h,s+h} \cap \clint{0,T}) = \D\ell_y(\{s\}) = 0$ for every $s \in C$. Let us recall that for any Banach space $E$ and any $f \in \L^1(\D\ell_y; E)$ there exists a $\D\ell_y$-zero measure set $Z$ such that $f(\clint{0,T} \setmeno Z)$ is separable (see, e.g., \cite[Property M11, p. 124]{Lan93}), therefore from 
\cite[Corollary 2.9.9., p. 155]{Fed69} it follows that 
\begin{equation}\label{Dl-Lebesgue points}
  \lim_{h \searrow 0} \frac{1}{\ell_y(s+h)-\ell_y(s-h)} \int_{\clsxint{s-h,s+h} \cap \clint{0,T}} \norm{f(t) - f(s)}{E} \de \D\ell_y(t) 
  = 0
\end{equation}
for $\D\ell_y$-a.e. $s\in C$. In \cite{Fed69} the points $s$ satisfying \eqref{Dl-Lebesgue points} are called
\emph{$\D\ell_y$-Lebesgue points of $f$ on $C$ with respect to the Vitali relation 
$V = \{\clsxint{s-h,s+h} \cap C\ ;\ s \in C,\ h > 0\}$}. Let $L$ be the set of $\D\ell_y$-Lebesgue points for both 
$t \longmapsto v(t)$ and $t \longmapsto \duality{u(t) - y(t)}{v(t)}$ on $C$ with respect to $V$, thus 
$\D\ell_y(C \setmeno L) = 0$. Now fix $s \in L$ and $\zeta \in \Z$. A straighforward computation shows that 
\[ 
  \lim_{h \searrow 0}  
  \frac{1}{\ell_y(s+h)-\ell_y(s-h)} \int_{\clsxint{s-h,s+h} \cap \clint{0,T}} \duality{\zeta}{v(t)} \de \D\ell_y(t) = 
  \duality{\zeta}{v(s)}.
\]
Taking $z(t) := \zeta \indicator_{\clsxint{s-h,s+h}}(t) + (u(t) - y(t)) \indicator_{\clint{0,T} \setmeno \clsxint{s-h,s+h}}(t)$ in
\eqref{var in KrLa} for $h > 0$ sufficiently small we get
\[
  \int_{\clsxint{s-h,s+h} \cap \clint{0,T}} \duality{\zeta}{v(t)} \de \D\ell_y(t) \le 
  \int_{\clsxint{s-h,s+h}\cap \clint{0,T}} \duality{u(t) - y(t)}{v(t)}\de \D \ell_y(t)
\] 
Dividing this inequality by $\ell_y(s+h)-\ell_y(s-h)$ and taking the limit as $h \searrow 0$ we get
$\duality{\zeta - u(s) + y(s)}{v(s)} \le 0$, therefore
\begin{equation}\label{diff-incl. on cont points}
  \duality{\zeta - u(s) + y(s)}{v(s)} \le 0 \qquad \text{for $\D\ell_y$-a.e. $s \in C$.}
\end{equation}
Now let $s \in \discont(\ell_y)$ and take 
$z(t) =  \zeta \indicator_{\clsxint{s,s+h}} (t) + (u(t) - y(t)) \indicator_{ \clint{0,T} \setmeno \clsxint{s,s+h} }(t)$ in
\eqref{var in KrLa}: we get
\[
  \int_{\clsxint{s,s+h}} \duality{\zeta - u(t) + y(t)}{v(t)}\de \D\ell_y(t) \le 0
\]
and taking the limit as $h \searrow 0$, by the dominated convergence theorem we infer that
\begin{equation}\notag
  0 \ge \int_{\{s\}} \duality{\zeta - u(t) + y(t)}{v(t)}\de \D\ell_y(t) = \duality{\zeta - u(s) + y(s)}{v(s)} \D\ell_y(\{s\}),
\end{equation}
hence
\begin{equation}\label{diff-incl. on jump points}
  \duality{\zeta - u(s) + y(s)}{v(s)} \le 0 \qquad \forall s \in \discont(\ell_y).
\end{equation}
Collecting together \eqref{diff-incl. on cont points}--\eqref{diff-incl. on jump points} and the fact that 
$\D\ell_y(\cont(\ell_y) \setmeno C) = 0$, we get \eqref{diff. incl. - BVsweep} and we are done.
\end{proof}


\section{Reduction to Lipschitz sweeping processes}\label{S:reduction}

Within this section we consider $u$ $\in$ $\BV^\r(\clint{0,T};\H)$ and the moving convex set $\C_u(t) = u(t) - \Z$, and we study the properties of the sweeping process driven by the reparametrized curve $\Ctilde_u \in \Lip(\clint{0,T};\Conv_\H)$. In this way we will be able to get information on the play operator thanks to the formula 
$\P(z_0, u) = \M(u(0) - z_0, \C_u) = \M(u(0) - z_0, \Ctilde) \circ \ell_\C$. It is useful to introduce the operators
\[
  \Stop : \Z \times \BV^\r(\clint{0,T};\H) \function \BV^\r(\clint{0,T};\H), \qquad 
  \Q : \Z \times \BV^\r(\clint{0,T};\H) \function \BV^\r(\clint{0,T};\H),
\]
defined by 
\begin{equation}\label{def. of S and Q}
  \Stop(z_0, u) := u - \P(z_0, u), \quad \Q(z_0, u) := \P(z_0, u) - \Stop(z_0, u), \qquad u \in \BV^\r(\clint{0,T};\H).
\end{equation}
In the regular case, the derivatives of these operators have a useful geometric interpretation, indeed if $z_0 \in \Z$ and
$u \in \W^{1,1}(0,T;\H)$ then it is easily seen (cf. \cite[Proposition 3.9, p. 33]{Kre97}) that 
$\duality{(\Stop(z_0, u))'}{(\Q(z_0, u))'} = 0$ for $\leb^1$-a.e. $t \in \clint{0,T}$, hence $(\Q(z_0, u))'(t)$ and $u'(t)$ are the diagonals of the rectangle with sides $(\Stop(z_0, u))'(t)$ and $(\P(z_0, u))'(t)$: it follows that 
$\norm{(\Q(z_0, u))'(t)}{} = \norm{u'(t)}{}$ for $\leb^1$-a.e., $t \in \clint{0,T}$ and this is a fundamental fact in the proof of the $\BV$-continuity of the play  operator in $\W^{1,1}(\clint{0,T};\H)$. Such relation makes no sense in the $\BV$ framework, but we will see that the operators $\Stop$ and $\Q$ still play a role.

\begin{Lem}
Let $\C_u$ be defined by \eqref{def Cu} for every $u \in \BV^\r(\clint{0,T};\H)$, and 
$\Q : \Z \times \BV^\r(\clint{0,T};\H)$ $\function$ $\BV^\r(\clint{0,T};\H)$ by \eqref{def. of S and Q}, i.e.
\begin{equation}
  \Q(z_0, u) := 2\P(z_0, u) - u, \qquad z_0 \in \Z, \ u \in \BV^\r(\clint{0,T};\H).
\end{equation}
Then
\begin{equation}\label{ell u = ell Cu}
  \V(u,\clint{0,T}) = \V(\C_u,\clint{0,T}), \qquad \ell_u = \ell_{\C_u} \qquad \forall u \in \BV^\r(\clint{0,T};\H),
\end{equation}
and 
\begin{equation}
  \Q(z_0, u) = (2\M(u(0) - z_0, \Ctilde_u) - \utilde) \circ \ell_u \qquad \forall u \in \BV^\r(\clint{0,T};\H).
\end{equation}
\end{Lem}

\begin{proof}
Identity \eqref{ell u = ell Cu} follows from the fact that $\d_\hausd(u(t) - \Z, u(s) - \Z) = \norm{u(t) - u(s)}{}$ for every
$t, s \in \clint{0,T}$. If $u \in \BV^\r(\clint{0,T};\H)$ then from \eqref{def. play}, Theorem \ref{T:Lip-reduction of BVsweep}, 
\eqref{reparametrization1}, and \eqref{ell u = ell Cu} we infer that 
\begin{align*}
\Q(z_0, u) 
& = 2\P(z_0, u) - u = 2\M(u(0) - z_0, \C_u) - u = 2\M(u(0) - z_0, \Ctilde_u) \circ \ell_{\C_u} - \utilde \circ \ell_u  \\ 
& = 
2\M(u(0) - z_0, \Ctilde_u) \circ \ell_u - \utilde \circ \ell_u = (2\M(u(0) - z_0, \Ctilde_u) - \utilde) \circ \ell_u.
\end{align*}
\end{proof}

As a consequence, from Proposition \ref{Cn -> C => Ctilden -> Ctilde} we infer the following

\begin{Cor}\label{C-u n -> C-u}
Assume that $u, u_n \in \BV^\r(\clint{0,T};\H)$, $\C_u$ and $\C_{u_n}$ are defined as in \eqref{def Cu} for every 
$n \in \en$. If $\norm{u_n - u}{\BV} \to 0$, then $\d_{us}(\Ctilde_n,\Ctilde) \to 0$ as $n \to \infty$.
\end{Cor}

\begin{Lem}\label{|what'| = cost}
Assume that $u \in \BV^\r(\clint{0,T};\H)$, $\C_u$ is defined by \eqref{def Cu}, $z_0 \in \Z$, and set $y_0 := u(0) - z_0$. If
\begin{equation}
  w := \Q(z_0, u) := 2\P(z_0, u) - u
\end{equation}
and
\begin{equation}\label{def what}
  \what := 2\M(y_0, \Ctilde_u) - \utilde,
\end{equation}
then there exists a function $\hat{v}_w \in \L^{\infty}(0,T;\H)$ such that
\begin{itemize}
\item[(a)]
  $\hat{v}_w$ is a Lebesgue representative of $\what'$;
\item[(b)]  
  it holds
\begin{equation}\label{norm of what'}
  \norm{\hat{v}_w(\sigma)}{} = \norm{\utilde'(\sigma)}{} = \frac{\V(u,\clint{0,T})}{T}
  \qquad \text{for $\leb^1$-a.e. $\sigma \in \ell_u(\cont(u))$}
\end{equation}
(the case $\leb^1(\ell_u(\cont(u))) = 0$ is not excluded);
\item[(c)] 
  if $g_{w} : \clint{0,T} \function \H$ is defined by
\begin{equation}\label{density of what w.r.t. d ell}
  g_{w}(t) :=
  \begin{cases}
    \dfrac{\what(\ell_u(t)) - \what(\ell_u(t-))}{\ell_u(t) - \ell_u(t-)} & \text{if $t \in \discont(u)$}, \\
    \ \\
    \hat{v}_w(\ell_u(t)) & \text{otherwise},
  \end{cases}
\end{equation}
then
\begin{equation}
 \D w = \D\ \!(\what \circ \ell_u) = g_{w} \D \ell_u,
\end{equation}
i.e. $g_{w}$ is a density of $\D w = \D\ \!(\what \circ \ell_u)$ with respect to $\D\ell_u$.
\end{itemize}
\end{Lem}

\begin{proof}
If $\yhat := \M(y_0, \Ctilde_u)$ then
\begin{alignat}{3}
  & \yhat(\sigma) \in \Ctilde_u(\sigma) & \qquad & \forall \sigma \in \clint{0,T}, \label{y in C - tilde} \\
  & \yhat'(\sigma) + \partial I_{\Ctilde_u(\sigma)}(y(\sigma)) \ni 0 & \qquad & 
     	\text{for $\leb^1$-a.e. $\sigma \in \clint{0,T}$}, 
     \label{diff. incl. C - tilde} \\
  & \yhat(0) = u(0) - x_0 \label{in. cond. C - tilde}
\end{alignat}
and, since it is immediately seen that
\begin{equation}\label{Ctilde on im(ell_u)}
  \Ctilde_u(\sigma) = \utilde(\sigma) - \Z \qquad \forall \sigma \in \ell_u(\clint{0,T}),
\end{equation}
it follows from \eqref{diff. incl. C - tilde} that
\begin{equation}\label{play part in Sw(Ctilde)}
  \duality{\yhat'(\sigma)}{z - \utilde(\sigma) + \yhat(\sigma)} \le 0 \qquad 
  \text{for $\leb^1$-a.e. $\sigma \in \ell_u(\clint{0,T})$}
\end{equation}
(the case $\leb^1(\ell_u(\clint{0,T})) = 0$ is not excluded). Let $A$ be the set where $\what$ is differentiable, hence 
$\leb^1(\clint{0,T} \setmeno A) = 0$, and observe that \eqref{solution for catching-up geodesic-2} and \eqref{def what} imply that $\what$ is affine on every interval of the form $\opint{\ell_u(t-), \ell_u(t)}$ with $t \in \discont(u)$, thus 
$B := \bigcup_{t \in \discont(u)} \opint{\ell_u(t-), \ell_u(t)} \subseteq A$. Now define $C$ as the set of points 
$\sigma \in A \cap \ell_u(\cont(u))$ such that there are two sequences $h_n, k_n \in \ar$ such that $h_n \searrow 0$ and $k_n \searrow 0$ as $n \to \infty$ and $\sigma + h_n \in \ell_u(\clint{0,T})$ and $\sigma - k_n \in \ell_u(\clint{0,T})$ for every $n \in \en$. Let us notice that $C \cap B = \void$ and take $z = \xhat(\sigma + h_n)$  (respectively 
$z = \xhat(\sigma - h_n)$) in \eqref{play part in Sw(Ctilde)}, divide by $h_n$ (resp. by $k_n$), and take the limit as 
$n \to \infty$: as a result we get $\duality{\yhat'(\sigma)}{\utilde'(\sigma) - \yhat'(\sigma)} = 0$. Therefore for every 
$\sigma \in C$ we have
\begin{align}
 \norm{\what'(\sigma)}{}^2 
   & = \norm{\yhat'(\sigma) - (\utilde'(\sigma) - \yhat'(\sigma))}{}^2 
      = \norm{\yhat'(\sigma)}{}^2 + \norm{\utilde'(\sigma) - \yhat'(\sigma)}{}^2 \notag \\
   & = \norm{\yhat'(\sigma) + (\utilde'(\sigma) - \yhat'(\sigma))}{}^2  = \norm{\utilde'(\sigma)}{}^2, \notag        
\end{align}
i.e.
\begin{equation}
  \norm{\what'(\sigma)}{} = \norm{\utilde'(\sigma)}{}
  \qquad \forall \sigma \in C.
\end{equation}
Now let $\sigma \in D:= \left(A \cap \ell_u(\cont(u))\right) \setmeno C$. From \eqref{ell(0,T)=} it follows that $\sigma$ is the endpoint of an interval of the kind $\opint{\ell_u(t-),\ell_u(t)}$ with $t \in \discont(u)$, thus at most two possibilities can occur:
\begin{itemize}
\item[(a)] $\sigma \in \ell_u(\opint{t-\delta,t})$ with $t \in \discont(u)$, $\delta > 0$ and $\ell_u(s) = \ell_u(t-)$ 
             for every 
               $s \in \opint{t-\delta,t}$: therefore $\D \ell_u(\ell_u^{-1}(\sigma)) = 0$;
\item[(b)] $\sigma = \ell_u(t)$ with $t \in \discont(u)$ and $\ell_u^{-1}(\sigma) = \clsxint{t,s}$ with $s \in \discont(u)$ and
               $\ell_u$ constant on $\clsxint{t,s}$: therefore $\D \ell_u(\ell_u^{-1}(\sigma)) = 0$.
\end{itemize}
It follows that, since $(\opint{\ell_u(t-),\ell_u(t)})_{t \in \discont(u)}$ is a countable family, $\leb^1(D)= 0$ and 
$\D \ell_u(\ell_u^{-1}(D))$ $=$ $0$. Therefore if $e \in \H$ is such that $\norm{e}{} = 1$ (if $\H = \{0\}$ there is nothing to prove), then the function $\hat{v}_w : \clint{0,T} \function \H$ defined by
\begin{equation}
  \hat{v}_w(\sigma) :=
  \begin{cases}
    \what'(\sigma) & \text{if $\sigma \in B \cup C$}, \\
    \ \\
    \dfrac{\V(\utilde,\clint{0,T})}{T}e & \text{otherwise}
  \end{cases}
\end{equation}
satisfies the required properties. Now the last statement on $g_{w}$ follows from Proposition \ref{P:BV chain rule} and 
from the fact that $\D \ell_u(\ell_u^{-1}(D)) = 0$.
\end{proof}


\section{Proof of the main Theorem}\label{S:proof main thm}

In this section we prove the main Thereom \ref{T:main thm}. First, for the reader's convenience we restate the weak compactness theorem for measures \cite[Theorem 5, p. 105]{{DieUhl77}} in a form which is suitable to our purposes.
 
\begin{Thm}\label{measure dunford pettis}
Let $I \subseteq \ar$ be an interval and let $M$ be a subset of the vector space of measures $\mu : \borel(I) \function \H$ with bounded variation endowed with the norm $\norm{\mu}{} := \vartot{\mu}(I)$. Assume that $M$ is bounded. Then $M$ is weakly sequentially precompact if and only if  there exists a bounded positive measure 
$\nu : \borel(I) \function \clsxint{0,\infty}$ such that for every $\eps > 0$ there is a $\delta > 0$ which satisfies the implication
\begin{equation}
  \forall \eps > 0 \ \exists \delta> 0 \ \quad :\quad 
  \left( B \in \borel(I),\ \nu(B) < \delta\  \Longrightarrow \ \sup_{\mu \in M} \vartot{\mu}(B) < \eps \right).
\end{equation}
\end{Thm}

Theorem \ref{measure dunford pettis} is stated in \cite[Theorem 5, p. 105]{DieUhl77} as a topological precompactness result. An inspection in the proof easily shows that this is actually a sequential precompatness theorem, since an isometric isomorphism reduces it to the well-known Dunford-Pettis weak sequential precompactness theorem in 
$\L^1(\nu;\H)$ (see, e.g., \cite[Theorem 1, p. 101]{DieUhl77}).

The following lemma is a vector measure counterpart of a well-known weak derivative argument.

\begin{Lem}\label{Dwn -> Dw}
Let $I \subseteq \ar$ be an interval, $w, w_n \in \BV(I;\H)$ for every $n \in \en$, and $\mu : \borel(I) \function \H$ be a measure with bounded variation. If $w_n \to w$ uniformly on $I$ and $\D w_n \convergedeb \mu$, then $\D w = \mu$.
\end{Lem}

\begin{proof}
Let $\overline{w}$ and $\overline{w}_n$ be the extensions  of $w$ and $w_n$ to $\ar$ defined as in 
\eqref{extension to R}. We have that $\overline{w}_n \to \overline{w}$ uniformly on $\ar$ and $\D \overline{w}$ and 
$\D \overline{w}_n$ are Borel measures of bounded variation on $\ar$, concentrated on $I$. We also extend $\mu$ to the measure $\overline{\mu} : \borel(\ar) \function \H$ defined by $\overline{\mu}(B) := \mu(B \cap I)$, $B \in \borel(\ar)$, thus we have $\D\overline{w}_n \convergedeb \overline{\mu}$. Let $x \in \H$ and $\varphi \in \Czero^1_c(\ar;\ar)$. Then the mapping $\nu \function \duality{x}{\int_\ar \varphi(t) \de \nu}$ is a linear continuous functional on the space of Borel measures with bounded variation on $\ar$, therefore we have
\begin{equation}
  \lim_{n \to \infty} \lduality{x}{\int_{\ar} \varphi \de\D \overline{w}_n} = \lduality{x}{\int_{\ar} \varphi \de\overline{\mu}} \notag
\end{equation}
On the other hand we have
\begin{align}
  \lim_{n \to \infty} \lduality{x}{\int_{\ar} \varphi \de\D \overline{w}_n} 
   =  \lim_{n \to \infty} \lduality{x}{-\int_{\ar} \varphi'(t) \overline{w}_n(t) \de t}  
   = \lduality{x}{-\int_{\ar} \varphi'(t) \overline{w}(t) \de t} \notag
\end{align}
hence
\[
  \lduality{x}{\int_\ar \varphi \de\overline{\mu}} = \lduality{x}{-\int_{\ar} \varphi'(t) \overline{w}(t)\de t}
\]
and from the arbitrariness of $x$ it follows that 
\[
 \int_{\ar} \varphi \de\overline{\mu} = -\int_{\ar} \varphi'(t) \overline{w}(t)\de t = \int_{\ar} \varphi \de\D\overline{w},
\]
thus $\overline{\mu} = \D \overline{w}$ by the arbitrariness of $\varphi$. Hence $\mu = \D w$.
\end{proof}

We are now in position to provide the

\begin{proof}[Proof of Theorem \ref{T:main thm}]
Assume that $z_0, z_{0,n} \in \Z$, $u, u_n \in \BV^\r(\clint{0,T};\H)$ for every $n \in \en$ and that $z_{0,n} \to z_0$ and 
$\norm{u_n - u}{\BV} \to 0$ as $n \to \infty$. Let us set $y_{0} := u(0) - z_0$, $y_{0,n} := u_n(0) - z_{0,n}$ for every 
$n \in \en$. For simplicity we define $\C, \C_n \in \BV^\r(\clint{0,T};\Conv_\H)$ by $\C(t) := \C_u(t) = u(t) - \Z$, 
$\C_n(t) := \C_{u_n}(t) = u_n(t) - \Z$, $t \in \clint{0,T}$, and we set $\ell := \ell_u = \ell_\C$, 
$\ell_n := \ell_{u_n} = \ell_{\C_n}$ (cf. \eqref{ell u = ell Cu}) for every $n \in \en$. Hence Theorem 
\ref{T:Lip-reduction of BVsweep} yields
\begin{equation}\label{P(u) = ..., P(u n) = ...}
  \P(z_0, u) = \M(y_0, \Ctilde) \circ \ell, \quad \P(z_{0,n}, u_n) = \M(y_{0,n},\Ctilde_n) \circ \ell_n 
  \qquad \forall n \in \en.
\end{equation} 
We also define 
\begin{equation}\label{Q(u) = ..., Q(u n) = ...}
  w := \Q(z_0, u) = 2\P(z_0, u) - u, \qquad w_n := \Q(z_{0,n}, u_n) = 2\P(z_{0,n}, u_n) - u_n,
\end{equation} 
and 
\begin{equation}\label{def what n}
  \what := 2\M(y_0, \Ctilde) - \utilde, \qquad \what_n := 2\M(y_{0,n}, \Ctilde_{n}) - \utilde_n.
\end{equation} 
Now, with these notations, let $g_w \in \L^{\infty}(0,T;\H)$ and $g_{w_n} \in \L^{\infty}(0,T;\H)$ be the density functions provided by Lemma \ref{|what'| = cost} in formula \eqref{density of what w.r.t. d ell}, with $w$ replaced by $w_n$ in the case of $g_{w_n}$, and for simplicity set $g := g_w$, $g_n := g_{w_n}$. Therefore we have that
\begin{equation}
  \D w = g \D \ell, \qquad \D w_n = g_n \D\ell_n.
\end{equation}
We will prove that $\norm{w_n - w}{\BV} \to 0$ as $n \to \infty$, and the conclusion follows from 
\eqref{Q(u) = ..., Q(u n) = ...} and from the linearity of the $\BV$-norm topology. From \eqref{P(u) = ..., P(u n) = ...},
\eqref{Q(u) = ..., Q(u n) = ...}, the uniform convergence of $u_n$ to $u$, Corollary \ref{C-u n -> C-u}, and from the continuity property of $\M$ stated in Theorem \ref{T:existence general BVsweep}, we infer that
\begin{equation}
  \norm{w_n - w}{\infty} \to 0 \qquad \text{as $n \to \infty$}. \label{wn -> w unif}
\end{equation}
Moreover from the inequality $|\V(u_n,\clint{s,t}) - \V(u,\clint{s,t})| \le \V(u_n - u, \clint{s,t})$, $0 \le s \le t \le T$, and from the triangle inequality we immediately get that
\begin{equation}
 \vartot{\D\ \!(\ell_n - \ell)}(\clint{0,T}) = \V(\ell_n - \ell, \clint{0,T}) \to 0 \qquad \text{as $n \to \infty$}.
\end{equation}
Thanks to  \eqref{density of what w.r.t. d ell}, \eqref{def what n}, \eqref{estimate Lip sweep proc}, 
\eqref{Lip(ftilde) = Lip(F)}, and \eqref{ell u = ell Cu}, we have that for every $t \in \discont(u)$ and for every $n \in \en$  
\begin{align}
 \norm{g_{n}(t)}{}
    & \le \Lipcost(\what_n) \le 2\Lipcost(\M(y_0,\Ctilde_n)) + \Lipcost(\utilde) \notag \\
    & \le 2\Lipcost(\Ctilde_n) +  \Lipcost(\utilde)  \le 2\V(\C_n,\clint{0,T})/T + \V(u_n,\clint{0,T})/T \notag \\
    & = 3\V(u_n,\clint{0,T})/T, \notag
\end{align}
while from \eqref{density of what w.r.t. d ell} and \eqref{norm of what'} we infer that 
$\norm{g_{n}(t)}{} \le \V(u_n,\clint{0,T})$ for every $t \in \cont(u)$ and for every $n$. Hence there is a constant $C > 0$, independent of $n$, such that
\begin{equation}\label{bound for g wn}
  \norm{g_{n}(t)}{} \le C \qquad \forall n \in \en,
\end{equation}
and
\begin{equation}
  \vartot{\D w_n}(B) = \int_B \norm{g_{n}(t)}{} \de \D\ell(t) \le C \vartot{\D \ell_n}(B) \qquad
  \forall B \in \borel(\clint{0,T}).
\end{equation}
Therefore, since in particular $\D\ell_n$ is weakly convergent to $\D \ell$, by the weak sequential compactness 
Dunford-Pettis Theorem \ref{measure dunford pettis} for vector measures, by \eqref{wn -> w unif}, and by Lemma 
\ref{Dwn -> Dw}, we have that $\D w_n$ is weakly convergent to $\D w$, in particular if $\phi : \clint{0,T} \function \H$ is an arbitrary bounded Borel function then $\mu \longmapsto \int_{\clint{0,T}} \duality{\phi(t)}{\de \mu(t)}$ is a continuous linear functional on the space of measures with bounded variation and we have
\[
  \lim_{n \to \infty} \int_{\clint{0,T}} \duality{\phi(t)}{\de \D w_n(t)} = \int_{\clint{0,T}} \duality{\phi(t)}{\de \D w(t)},
\]
i.e.
\begin{equation}\label{gn Dwn -> gDw}
  \lim_{n \to \infty} \int_{\clint{0,T}} \duality{\phi(t)}{g_{n}(t)} \de \D \ell_n(t) = 
  \int_{\clint{0,T}} \duality{\phi(t)}{g(t)}\de \D \ell(t).
\end{equation}
On the other hand, by \eqref{bound for g wn}, we have that there exists $z \in \L^{p}(\D\ell;\H)$ such that 
$g_{n} \convergedeb z$ in $\L^p(\D\ell;\H)$ for every $p \in \opint{1,\infty}$, therefore if we set 
$\psi_n(t) := \duality{\phi(t)}{g_{n}(t)}$ and $\psi(t) := \duality{\phi(t)}{z(t)}$ for $t \in \clint{0,T}$, we have that 
$\psi_n \convergedeb \psi$ in $\L^p(\D\ell;\ar)$, $p \in \opint{1,\infty}$, thus
\begin{align}
   & \sp \left| \int_{\clint{0,T}} \psi_n(t) \de \D \ell_n(t) -  \int_{\clint{0,T}} \psi(t) \de \D \ell(t) \right| \notag \\
   &  \le  \int_{\clint{0,T}} |\psi_n(t)| \de \vartot{\D\ \!(\ell_n - \ell)}(t) +
       \left| \int_{\clint{0,T}} (\psi_n(t) - \psi(t)) \de \D \ell(t) \right| \notag \\
   & \le  \norm{\phi}{\infty}\norm{g_{n}}{\infty} \vartot{\D\ \!(\ell_n - \ell)}(\clint{0,T}) +
       \left| \int_{\clint{0,T}} (\psi_n(t) - \psi(t)) \de \D \ell(t) \right| \notag \\
   &  \le  C \norm{\phi}{\infty}\norm{u_n - u}{\BV} +
       \left| \int_{\clint{0,T}} (\psi_n(t) - \psi(t)) \de \D \ell(t) \right| \to 0 \notag
\end{align}
as $n \to \infty$. This means that
\[
  \lim_{n \to \infty} \int_{\clint{0,T}} \duality{\phi(t)}{g_{n}(t)} \de \D \ell_n(t)\ = 
  \int_{\clint{0,T}} \duality{\phi(t)}{z(t)}\de \D \ell(t),
\]
hence, by \eqref{gn Dwn -> gDw},
\begin{equation}\label{g dl = z dl weakly}
  \int_{\clint{0,T}} \duality{\phi(t)}{\de(g\D \ell)(t)} = 
  \int_{\clint{0,T}} \duality{\phi(t)}{\de (z\D \ell)(t)}.
\end{equation}
The arbitrariness of $\phi$ and \eqref{g dl = z dl weakly} imply that $z\D \ell = g\D \ell$ (cf. 
\cite[Proposition 35, p. 326]{Din67}), hence $z(t) = g(t)$ for $\D\ell$-a.e. $t \in \clint{0,T}$ and we have found that
\begin{equation}
  g_{n} \convergedeb g \qquad \text{in $\L^p(\D\ell;\H)$, \quad $\forall p \in \cldxint{1,\infty}$}. 
\end{equation}
Now observe that \eqref{density of what w.r.t. d ell} and \eqref{norm of what'} yield
\begin{equation}\label{|g_n| -> |g| pointwise on cont(u)}
  \lim_{n \to \infty} \norm{g_n(t)}{} = \lim_{n \to \infty} \frac{V(u_n,\clint{0,T})}{T} = \frac{V(u,\clint{0,T})}{T}
  = \norm{g(t)}{} \qquad \forall t \in \cont(u).
\end{equation}
Moreover $\ell_{n} \to \ell$ uniformly by \cite[Proposition 5.2]{Rec16}, while formula \eqref{def what n}, Corollary 
\ref{C-u n -> C-u}, and Proposition \ref{P:vntilde -> vtilde} imply that $\norm{\what_n - \what}{\infty} \to 0$ as $n \to\infty$, thus
\begin{equation}\label{|g_n| -> |g| pointwise on discont(u)}
  \lim_{n \to \infty} \left\|\frac{ \what_n(\ell_n(t)) - \what_n(\ell_n(t-)) }{\ell_n(t) - \ell_n(t-)}\right\| =
  \left\|\frac{ \what(\ell(t)) - \what(\ell(t-)) }{\ell(t) - \ell(t-)}\right\| \qquad \forall t \in \discont(u).
\end{equation}
From \eqref{|g_n| -> |g| pointwise on cont(u)}, \eqref{|g_n| -> |g| pointwise on discont(u)}, and \eqref{bound for g wn}
it follows that 
\begin{align}
\lim_{n \to \infty} \norm{g_{n}}{\L^p(\D\ell;\H)}^p
    & = \lim_{n \to \infty} \int_{\clint{0,T}} \norm{g_{n}(t)}{}^p \de \D\ell(t) \notag \\
    & = \int_{\clint{0,T}} \norm{g(t)}{}^p \de \D\ell(t) = \norm{g}{\L^p(\D\ell;\H)}^p \qquad \forall p \in \opint{1,\infty}, \notag 
\end{align}
therefore by the uniform convexity of $\L^p(\D\ell;\H)$ for $p \in \opint{1,\infty}$ we have
\begin{equation}
 g_{n}  \to g \qquad \text{in $\L^p(\D\ell;\H)$ $\forall p \in \opint{1,\infty}$},
\end{equation}
and, since $\D\ell(\clint{0,T}) = T < \infty$,
\begin{equation}
  g_{n} \to g \qquad \text{in $\L^1(\D\ell;\H)$}.
\end{equation}
Hence $g_n$ has a subsequence, which we do not relabel, that is convergent to $g$ for $\D\ell$-a.e. $t$, thus
\begin{align}
 V(w_n - w, [0,T])=  \norm{\D\ \!(w_n - w)}{}
    & =  \norm{\D w_n - \D w}{} =
            \norm{g_{n} \D \ell_n - g \D \ell}{} \notag \\
    & \le \norm{g_{n}\D\ \!(\ell_n - \ell)}{} + \norm{(g_{n} -g) \D\ell}{} \notag \\
    & \le C \norm{\D\ \!(\ell_n - \ell)}{} + 
            \int_{\clint{0,T}}\norm{g_{n}(t) - g(t)}{} \de \D\ell(t) \to 0\notag
\end{align}
as $n \to \infty$ and we are done.
\end{proof}



\end{document}